\documentclass[12pt,reqno]{amsart}

\usepackage[dvipsnames]{xcolor}
\usepackage{tikz,tikz-cd}
\usepackage{nicefrac}
\usepackage{adjustbox}
\usetikzlibrary{arrows,shapes,calc,plotmarks,decorations.markings}
\usepackage{hyperref}
\hypersetup{
	colorlinks,
	linkcolor={red!50!black},
	citecolor={blue!50!black},
	urlcolor={blue!80!black}
}
\usepackage{mathtools}
\usepackage[margin = 1in]{geometry}
\usepackage{amsfonts}
\usepackage{amsmath}
\usepackage{amsthm}
\usepackage{multirow}
\usepackage[capitalize,noabbrev]{cleveref}

\newtheorem{theorem}{Theorem}
\numberwithin{theorem}{section}

\newtheorem{lemma}[theorem]{Lemma}
\newtheorem{proposition}[theorem]{Proposition}
\newtheorem{corollary}[theorem]{Corollary}

\theoremstyle{remark}
\newtheorem{example}[theorem]{Example}

\theoremstyle{remark}
\newtheorem{remark}[theorem]{Remark}

\theoremstyle{remark}

\theoremstyle{definition}
\newtheorem{definition}[theorem]{Definition}

\newcommand{\mz}{\mathbb{Z}}
\newcommand{\mr}{\mathbb{R}}
\newcommand{\mn}{\mathbb{N}}
\newcommand{\vect}{\textbf{vect}}
\newcommand\bZ{{\mathbb{Z}}}
\newcommand{\R}{\mathbf{R}}
\newcommand{\Z}{\mathbf{Z}}
\newcommand{\dgm}{[m+1]^2_<}
\newcommand{\Dgm}{\R^2_<}

\newcommand\abs[1]{\lvert#1\rvert}
\newcommand\norm[1]{\left\lVert#1\right\rVert}
\newcommand{\DGM}{\mathcal{D}}

\newcommand{\eps}{\varepsilon}

\DeclareMathOperator{\intc}{Int}
\DeclareMathOperator{\intp}{Int}
\DeclareMathOperator{\rank}{rank}
\DeclareMathOperator{\Rank}{Rank}
\DeclareMathOperator{\ptwise}{\Sigma}
\DeclareMathOperator{\pd}{PD}
\DeclareMathOperator{\op}{op}

\begin{document}
	\title{Graded Persistence Diagrams and Persistence Landscapes}
	
	\author{Leo Betthauser}
	\address{Microsoft Corporation}
	\email{lebettha@microsoft.com}
	
	\author{Peter Bubenik}
	\address{Department of Mathematics, University of Florida}
	\email{peter.bubenik@ufl.edu}
	
	\author{Parker B. Edwards}
	\address{ \parbox{\linewidth}{Department of Applied and Computational Mathematics and Statistics, \\ University of Notre Dame} }
	\email{parker.edwards@nd.edu}
	
	
	\date{\today}


\begin{abstract}
  We introduce a refinement of the persistence diagram, the graded persistence diagram. It is the M\"obius inversion of the graded rank function, which is obtained from the rank function using the unary numeral system. Both persistence diagrams and graded persistence diagrams are integer-valued functions on the Cartesian plane. Whereas the persistence diagram takes non-negative values, the graded persistence diagram takes values of $0$, $1$, or $-1$. The sum of the graded persistence diagrams is the persistence diagram. We show that the positive and negative points in the $k$-th graded persistence diagram correspond to the local maxima and minima, respectively, of the $k$-th persistence landscape. We prove a stability theorem for graded persistence diagrams: the $1$-Wasserstein distance between $k$-th graded persistence diagrams is bounded by twice the $1$-Wasserstein distance between the corresponding persistence diagrams, and this bound is attained. In the other direction, the $1$-Wasserstein distance is a lower bound for the sum of the $1$-Wasserstein distances between the $k$-th graded persistence diagrams. In fact, the $1$-Wasserstein distance for graded persistence diagrams is more discriminative than the $1$-Wasserstein distance for the corresponding persistence diagrams.
\end{abstract}
	
\maketitle	
	

	\section{Introduction}
	
	In computational settings, persistent homology produces a persistence module indexed by the ordered set $[m] = \{0,1,2,\ldots,m\}$.
	For each persistence module there is a rank function giving the ranks of the linear maps corresponding to the pairs $a \leq b$, where $a,b \in [m]$.
	The persistence diagram of such a persistence module was first defined by Cohen-Steiner, Edelsbrunner, and Harer~\cite{cseh:stability}. It is obtained from the rank function using a simple inclusion-exclusion formula, and the rank function may be recovered using summation. Patel observed that this is an example of M\"obius inversion~\cite{Patel:2018}. 
	An alternative summary of persistence modules is the persistence landscape~\cite{bubeniklandscapes}. It may be viewed as a feature map or kernel~\cite{Reininghaus:2015,Bubenik:pl-properties}, allowing methods from machine learning and statistics to be easily applied to persistence modules.
	
	Here we show that there is an elegant connection between these two approaches. The key step uses the simplest (and surely the oldest) way of representing natural numbers: the unary numeral system. We decompose the rank function into a sequence of $k$-th graded rank 
	functions, for $k \in \mn$, whose values lie in $\{0,1\}$. M\"obius inversion produces the $k$-th graded persistence diagram. Unlike the persistence diagram, whose values lie in $\mz_{\geq 0}$, its values lie in $\{-1,0,1\}$. The sum of the graded persistence diagrams is the persistence diagram (\cref{thm:consistency}), so it is a refinement of the usual construction. Furthermore, the points where the $k$-th graded persistence diagram has values of $1$ and $-1$ are the local maxima and local minima, respectively, of the $k$-th persistence landscape (\cref{thm:properties-pl}).
	Using the graded persistence diagram, we give a simple definition of the derivative of the persistence landscape (\cref{def:rhok}, \cref{prop:rhok-lambdak}). 
	
	In our development, we carefully define persistence modules, rank functions, and persistence diagrams in both the discrete and continuous cases so that the constructions are compatible (Figures \ref{fig:pm} and \ref{fig:rank}).
	
	A $1$-Wasserstein distance may be defined for persistence diagrams whose points are allowed to have negative multiplicity~\cite{be:virtual}. We follow this idea to define a $1$-Wasserstein distance for graded persistence diagrams. For $p>1$ the $p$-Wasserstein distance for graded persistence diagrams does not satisfy the triangle inequality (\cref{prop:triangle}).	We prove the following stability theorem:
	The $1$-Wasserstein distance between two $k$-th graded persistence diagrams is at most twice the $1$-Wasserstein distance between their corresponding persistence diagrams, and this upper bound is achieved (\cref{thm:stability}). To the authors' knowledge, this is the first stability result for generalized persistence diagrams with negative multiplicity.
	We also give sharp bounds for the sum of the $1$-Wasserstein distances between the $k$-th graded persistence diagrams in terms of the $1$-Wasserstein distance between the corresponding persistence diagrams (\cref{thm:bounds}).

	For two metrics $d$, $d'$ on a set $X$, say that $d$ is \emph{more discriminative} than $d'$ if $d'(x,y) \leq d(x,y)$ for all $x,y \in X$ and if there is no constant $M$ such that $d(x,y) \leq Md'(x,y)$ for all $x,y \in X$.
	For example, for the set of tame persistence modules, for $1 \leq p < q \leq \infty$ the $p$-Wasserstein distance of their persistence diagrams is more discriminative than the $q$-Wasserstein distance.
	By \cref{thm:bounds}, for the set of tame persistence modules, the $1$-Wasserstein distance of their graded persistence diagrams is more discriminative than the $1$-Wasserstein distance of their persistence diagrams.

	As a result of our theory, algorithms and software for the graded persistence diagram are already available. Indeed, the standard software for computing persistence landscapes~\cite{bubenikDlotko} stores the piecewise-linear $k$-th persistence landscape by its critical points, which is the $k$-th graded persistence diagram.
	
	\subsection*{Related work}
	
	Patel~\cite{Patel:2018} uses M\"obius inversion to define and study persistence diagrams of constructible persistence modules indexed by $\mathbb{R}$ with values in certain symmetric monoidal categories and certain abelian categories. In that latter case, he proves a stability theorem for erosion distance.
	Patel and McCleary \cite{McCleary:2018} strengthen this to a bottleneck-distance stability theorem.
	More recently, they study persistence modules indexed by $\mathbb{R}^n$ and prove bottleneck stability under the assumption that all of the elements in the persistence diagram are positive~\cite{McCleary:2019b}.
	Puuska \cite{Puuska:2017} has generalized Patel's erosion stability result to the setting of generalized persistence modules~\cite{bubenik2015metrics}.
	Memoli and Kim \cite{Kim:2019} define a notion of rank invariant for persistence modules indexed by a poset with values in certain symmetric monoidal categories. When the posets are essentially finite, they use this rank invariant to define persistence diagrams which they use to study zigzag persistence and Reeb graphs.
	Vipond \cite{Vipond:2018} generalizes persistence landscapes~\cite{bubeniklandscapes} to define and study persistence landscapes for persistence modules indexed by $\mathbb{R}^n$. We note that in the previous two cases \cite{Kim:2019,Vipond:2018} the persistence diagrams obtained by M\"obius inversion may have negative terms like the graded persistence diagrams studied here.
	Inspired by these persistence diagrams with negative terms, Bubenik and Elchesen \cite{be:virtual} have undertaken a more systematic study of such diagrams.

	\subsubsection*{Outline of the paper}
	
	In \cref{sec:background} we provide background on persistence modules, the rank function, persistence landscapes and M\"obius inversion, including a careful construction of compatible discrete and continuous persistence modules and rank functions. In \cref{sec:intervals-pd} we show how to apply M\"obius inversion to the rank function on half-open intervals to obtain a persistence diagram.
	In \cref{sec:gpd} we define the graded rank functions and apply M\"obius inversion to obtain the graded persistence diagrams. Compatibility with the usual approach is given in our Consistency Theorem (\cref{thm:consistency}).
	Using the graded rank function, we define and characterize the persistence landscape (\cref{def:pl} and \cref{thm:properties-pl}). We also give a simple definition of the derivative of the persistence landscape in terms of the graded persistence diagram (\cref{def:rhok} and \cref{prop:rhok-lambdak}).
	In \cref{sec:wasserstein} we define $1$-Wasserstein distance for graded persistence diagrams (\cref{def:graded-wasserstein}) and use it to prove a stability theorem for graded persistence diagrams (\cref{thm:stability})
	and to give sharp bounds for the sum of the $1$-Wasserstein distances between the $k$-th graded persistence diagrams (\cref{thm:bounds}).

	\section{Background} \label{sec:background}
	
	In this section we introduce the background necessary for the subsequent sections. In particular, we introduce persistence modules, the rank function, persistence landscapes and M\"obius inversion. \cref{sec:discrete-cont} discusses persistence modules indexed by a real parameter obtained from persistence modules indexed by a finite set.
	
	\subsection{Partially ordered sets, intervals, and categories}
	
	A partially ordered set or poset $(P,\leq)$ is a set $P$ with a reflexive, transitive, and antisymmetric relation $\leq$.
	This poset will usually be denoted by $P$.
	A morphism of posets $f: P \to Q$ is an order-preserving map.
	We may also think of a poset $P$ as a category
	with objects the elements of $P$ and arrows $a\to b$ if and only if $a\leq b$.
	We may also think of a poset map $f:P \to Q$ as a functor between the corresponding categories.
	Let $P^{\op}$ denote the underlying set of $P$ together with the opposite order. That is $a \leq b$ in $P^{\op}$ if and only $b \leq a$ in $P$.
	An order-reversing map is a poset map $f: P^{\op} \to Q$.

	\begin{definition} 
		For $a\leq b$ in a poset $(P,\leq)$, the \emph{interval} $[a,b]$ is the set $\{ z \in P \mid a \leq z \leq b \}$. Denote the set of intervals in $P$ by $\intp(P)$.   
		Note that all intervals are nonempty by definition and that for each $a \in P$ there is an interval $[a,a]$ which contains only the element $a$.
		The set $\intp(P)$ is a poset with the partial order $\subset$ given by subset containment. That is, $[a,b] \subset [a',b']$ holds if and only if $a' \leq a \leq b \leq b'$ does. Given $f:\intc(P) \to Q$, for brevity we write $f([a,b])$ as $f[a,b]$.
	\end{definition}
	
	\begin{example} Consider the posets $[m] := \{ 0 < \dots < m \}$ and
		$\mathbf{R} = (\mr,\leq)$ and their corresponding posets of intervals $\intp([m])$ and $\intp(\mathbf{R})$.
	\end{example}

	\subsection{Persistence modules and rank functions} \label{sec:pm-rank}
	
	Let $K$ be a field and let $P$ be a sub-poset of $\R$. 
	A \emph{persistence module} $M$ with indexing poset $P$ assigns a finite dimensional vector space over $K$, $M(x)$, to every element $x\in P$ and a $K$-linear map $M(x \leq y):M(x)\to M(y)$ to every pair $x\leq y$ in $P$. The maps $M(x\leq y)$ for $x\leq y$ in $P$ satisfy $M(x\leq x) = 1_{M(x)}$ and $M(x\leq y) = M(z\leq y) \circ M(x\leq z)$ for all $z$ with $x \leq z \leq y$. Equivalently, $M$ is a functor $M:{P}\to\vect_{K}$, where $\vect_{K}$ denotes the category of finite dimensional $K$-vector spaces and $K$-linear maps. Persistence modules, particularly with indexing posets $\R$ and $[m]$, are central objects of study in persistent homology.

	\begin{definition} \label{def:rank}
		The \emph{rank function} of a persistence module $M$ with indexing poset $P$ is the function $\rank(M):\intc(P)\to\mz$ given by $\rank(M)([a,b]) = \rank(M(a \leq b))$.
		We will often omit $M$ and only write $\rank$.
	\end{definition} 
	
	The following theorem follows from the classification of persistence modules, which follows from the classification of graded modules over a graded PID \cite{zomorodianCarlsson:computingPH} or from Gabriel's classification of finite type quiver representations \cite{Gabriel1972}.
	
	\begin{theorem}
		Persistence modules $M$ and $N$ with indexing poset $[m]$ are naturally isomorphic if and only if $\rank(M) = \rank(N)$. \label{thm:gabriel}
	\end{theorem}
	
	\begin{lemma}\label{lem:rank_order} 
		For any persistence module $M$ with indexing poset $P$, the rank function $\rank(M):\intc(P) \to(\mz,\leq)$ is an order-reversing function, where $\leq$ is the standard order on $\mz$. 
	\end{lemma} 
	\begin{proof} 
		If $[x',y']$ and  $[x,y]$ are intervals in $\intc(P)$ with $[x',y'] \subset [x,y]$ then the following diagram commutes: 
		\[
		\begin{tikzcd}  
		M(x) \arrow[bend right=20,swap]{rrr}{M(x \leq y)} \arrow{r}{M(x\leq x')}&
		M(x') \arrow{r}{M(x' \leq y')} &
		M(y') \arrow{r}{M(y' \leq y)} & 
		M(y) \text{.}
		\end{tikzcd}
		\]
		Since $M(x\leq y)$ factors through $M(x' \leq y')$, it follows that $\rank(M(x\leq y))$ is at most $\rank(M(x' \leq y'))$. \end{proof} 
	
	Let $\Z_+$ denote the poset $(\mz_{\geq 0}, \leq)$. Then \cref{lem:rank_order} says that we have a poset morphism (i.e. an order-preserving map)
	$\rank(M): \intc(P)^{\op} \to \Z_+$.

	\subsection{Discrete and continuous persistence modules}
	\label{sec:discrete-cont}
	
	For computations, we are primarily interested in persistence modules indexed by $[m]$ for some $m \in \mn$. For applications, the underlying parameter is often continuous and we are interested in persistence modules indexed by $\R$.

	We will assume that our object of study is a persistence module $M$ indexed by $\R$ but that we have only finitely many observations and that these completely determine the persistence module.
	That is, $M$ is completely determined (up to isomorphism) by the vector spaces $M(a_i)$  and linear maps $M(a_i \leq a_j)$ for finitely many parameter values $a_0,\dots,a_m$ with $a_0 < a_1< \cdots <a_m$. Such persistence modules are sometimes referred to as \emph{tame}, \emph{finite type}, or \emph{constructible}.
	Specifically, we assume that there exist $m \in \mn$ and $a_0,\ldots,a_{m+1} \in \mathbb{R}$ such that $a_0 < a_1 < \cdots < a_m < a_{m+1}$ and that for all $i \in [m]$ and $a,b \in [a_i,a_{i+1})$ with $a \leq b$, the map $M(a \leq b)$ is an isomorphism and that $M(a) = 0$ for $a < a_0$ and for $a \geq a_{m+1}$.\footnote{We consider $a_{m+1}$ as a parameter value for which the experiment was terminated and lacking additional information we conservatively assume that nothing persists beyond this value. If desired, this value may be taken to be $\infty$.}
	For example, such persistence modules may arise from the homology of sublevel sets of a Morse function on a compact manifold.
	All persistence modules of this form arise from the following construction.
	
	Let $M$ be a persistence module indexed by $[m]$.
	Extend this to a persistence module $\hat{M}$ indexed by $[m+1]$ by defining $\hat{M}(m+1) = 0$.
	Let $\iota:[m+1] \to \R$ be an injective order-preserving map.
	For example, $\iota(j) = j$ for all $j \in [m+1]$.
	Then $\hat{M}$
	extends uniquely (up to isomorphism) to a persistence module $\overline{M}$ on $\R$ with $\overline{M}(\iota(j)) = \hat{M}(j)$ for $j \in [m+1]$ and $\overline{M}$ satisfies our assumption.
	See \cref{fig:pm}.
	In categorical language, $\hat{M}$ is the right Kan extension of $M$ along the inclusion map, and $\overline{M}$ is the left Kan extension of $\hat{M}$ along $\iota$.
	
	\begin{figure}
		\[
		\begin{tikzcd}
		{[m]} \arrow[r,"M"] \arrow[d] & \vect_{K}\\
		{[m+1]} \arrow[ur,"\hat{M}",dashed] \arrow[d,"\iota"',hook] \\
		\R \arrow[uur,"\overline{M}"',dashed]
		\end{tikzcd}
		\]
		\caption{Given persistence module indexed by $[m]$ and an injective map $\iota: [m+1] \to \R$ we have canonical extensions to persistence modules indexed by $[m+1]$ and $\R$. }
		\label{fig:pm}
	\end{figure}
	
	\subsection{Persistence landscapes} \label{sec:pl-background}
	
	Persistence landscapes were introduced for persistence modules with indexing poset $\R$~\cite{bubeniklandscapes}. Given such a module $\overline{M}$, its \emph{persistence landscape} is the function $\lambda:\mn\times\mr \to \mr$ given by 
	\[
	\lambda(k,t) = \sup\{z > 0 \mid \rank(\overline{M})([t-z,t+z]) \geq k \}\text{,}
	\]
	where $\lambda(k,t)=0$ if the set is empty.
	
	Each function $\lambda_k = \lambda(k,-):\mr\to\mr$ is a continuous piecewise-linear function with pieces of slope $+1$, $-1$, and $0$. In computational settings, each $\lambda_k$ has finitely many \emph{critical points} where the slope of the function changes, and there are finitely many $k$ for which $\lambda_k$ is not identically equal to zero. Computing and encoding a persistence landscape can be accomplished by identifying and storing the critical points of each $\lambda_k$~\cite{bubenikDlotko}.
	Additional properties of the persistence landscape may be found in subsequent papers~\cite{Bubenik:pl-properties,Chazal:2015c,Chazal:2015b}.

	\subsection{Incidence algebras and M\"obius inversion}
	
	In this section we recall some of the basic theory of M\"obius inversion for posets, which was initiated by Rota~\cite{rota1964foundations} and is an important part of enumerative combinatorics~\cite{StanleyEC}. This theory applies to posets that are \emph{locally finite}. A poset $P$ is \emph{locally finite} if for all pairs $x\leq y$ in $P$, the set $[x,y] = \{ p \mid x\leq p \leq y \}$ is finite.	The poset $[m]$ is locally finite, but $(\mr,\leq)$ is not. Fix a commutative ring $R$ with unit $1$ and a locally finite poset $(P,\leq)$. 
	
	\begin{definition}  \label{def:convolution}
		The \emph{convolution} operator is the following binary operator $*$ on the set of functions $\intc(P)\to R$.  For $f,g:\intc(P)\to R$ and interval $[x,y] \in \intc(P)$, 
		\[
		(f*g)[x,y] = \sum_{c \in [x,y]} f[x,c]g[c,y] \text{.}
		\]
		The \emph{incidence algebra} on $P$ consists of functions $\intc(P)\to R$ together with the convolution operator.
	\end{definition} 
	
	If $P$ has a largest element $\omega$, then for any function $h:P\to R$, identify $h$ with the function $h:\intc(P)\to R$ given by
	\[ 
	h[x,y] = \begin{cases} 
	h(x) & \text{ if } y = \omega \\ 
	0 & \text{ otherwise }
	\end{cases}
	\] 
	for all $x \leq y \in P$. Under this identification we have for $h:P\to R$ and $f:\intc(P) \to R$ that $f*h:P\to R$ is defined by
	\begin{equation} \label{eq:convolution}
		(f*h)(x) = (f*h)[x,\omega] = \sum_{x' \in [x,\omega]} f[x,x']h[x',\omega] = \sum_{x \leq x'} f[x,x']h(x') \text{.}
	\end{equation}

	The incidence algebra on $P$ contains the following three distinguished elements.
	
	\begin{definition} \label{def:functions}
		For any poset $P$ and commutative ring $R$ with unit $1$, define the following three functions $\intc(P)\to R$: 
		\begin{itemize}
			\item The \emph{zeta function} $\zeta_P:\intc(P)\to R$ has $\zeta_P(I) = 1$ for all $I\in \intp(P)$. 
			\item The \emph{delta function} $\delta_P:\intc(P)\to R$ has $\delta_P(I) = 1$ for all $I$ of the form $[x,x] \in \intc(P)$, and $\delta_P(I) = 0$ if $I$ is not of this form. 
			\item The \emph{M\"{o}bius function} $\mu_P:\intc(P)\to R$ is defined recursively as follows. For all $x\in P$, $\mu_P[x,x] = 1$, and for any $x < y$ define $\mu_P[x,y] = -\sum_{x \leq y' < y} \mu_P[x,y']$. 
		\end{itemize}
		We will drop the subscript $P$ from the functions above when the poset is clear from the context.
	\end{definition} 
	
	\begin{example}\label{ex:mob_inversion}
		Consider the partially ordered set $[m]$. For any $x\in [m]$ we can calculate from the above definition of $\mu_{[m]}$ that: 
		\[
		\mu[x,y] =  \begin{cases} 
		1 & \text{if } y = x \\
		-1 & \text{if } y = x+1 \\
		0 & \text{otherwise.}
		\end{cases}
		\]
		The calculation begins by noting $\mu[x,x] = 1$. Assume $x+1 \in [m]$. Then $\mu[x,x+1] = -\mu[x,x] = -1$ because $x$ is the only element of $[m]$ that is less than $x+1$ and greater than or equal to $x$. Subsequently note that if $x+2 \in [m]$ then $\mu[x,x+2] = -(\mu[x,x] + \mu[x,x+1]) = 0$, and similarly $\mu[x,y] = 0$ for all $y \geq x+2$ by induction. 
		
	\end{example}

	\begin{proposition}[{\cite[Chapter 3]{StanleyEC}}] \label{prop:identities} Consider the incidence algebra of a locally finite poset $P$ and commutative ring with identity $R$.
		\begin{enumerate}
			\item Convolution is associative.
			\item $\delta_P$ is an identity for convolution. That is $f*\delta_P = f = \delta_P*f$ holds for $f:\intc(P) \to R$. As a special case, $\delta_P*h = h$ for any $h:P\to R$.  
			\item The functions $\zeta_P$ and $\mu_P$ are inverses under convolution. That is, $\zeta_P * \mu_P = \delta_P = \mu_P * \zeta_P$. 
		\end{enumerate}
	\end{proposition} 
	
	\begin{example}
		Let $k \in [m]$.
		Consider the function $h:[m]\to\mz$ given by $h(i)=1$ if $i\leq k$ and $h(i)=0$ if $i > k$.
		From \eqref{eq:convolution} and \cref{ex:mob_inversion}, we have that for any $x \in [m]$
		\begin{align*} 
			(\mu*h)(x)
			& = \sum_{x \leq x' \leq m} \mu[x,x']h(x') \\
			& = h(x) - h(x+1) \quad \text{(where $h(m+1)=0$)}\\
			& =
			\begin{cases}
				1 & \text{if } x=k\\
				0 & \text{otherwise.}
			\end{cases}
		\end{align*} 
		Let $g = \mu*h$. For any $x\in [m]$,
		\begin{equation*} 
			(\zeta*g)(x)
			= \sum_{x \leq x' \leq m} \zeta[x,x']g(x') 
			= \sum_{x \leq x' \leq m} g(x') 
			= h(x).
		\end{equation*} 	
	\end{example}

	\section{Half-open intervals and persistence diagrams}
	\label{sec:intervals-pd}
	
	Consider a persistence module $M$ indexed by $[m]$ and an injective poset map $\iota:[m+1] \to \R$. Then we have a corresponding persistence module $\hat{M}$ indexed by $[m+1]$ and a persistence module $\overline{M}$ indexed by $\R$ as defined in \cref{sec:discrete-cont}. In this section we define compatible rank functions and persistence diagrams for $\hat{M}$ and $\overline{M}$ using half-open intervals.
	
	\subsection{Half-open intervals}
	\label{sec:half-open}
	
	In this section we define and discuss half-open intervals. For persistence modules indexed by $\R$ or by $[m]$, the support of a persistent homology class that is born at $a$ and that dies at $b$ is the half-open interval $[a,b)$.
	
	Let $P$ be a poset. For $a<b \in P$ define the \emph{half-open interval} $[a,b)$ to be the sub-poset of $P$ given by $\{c \in P \mid a \leq c < b\}$.
	Then the collection $\{[a,b) \mid a < b \in P\}$ is a poset with partial order given by inclusion.
	Call this the \emph{poset of half-open intervals} in $P$. The product poset $P^{\op} \times P$ consists of ordered pairs $(a,b)$ with $a,b \in P$ and the relation $(a,b) \leq (a',b')$ holds if and only if both $a' \leq a$ and $b \leq b'$ hold in $P$.
	Then the poset of half-open intervals may be identified with the sup-poset of $P^{\op} \times P$ given by $\{(a,b) \mid a < b\}$ under the mapping $[a,b) \mapsto (a,b)$.
	We denote the poset of half-open intervals in $P$ by $P^2_<$.
	\begin{example}
		For example, we have the posets of half-open intervals $\dgm$ and $\Dgm$.
	\end{example}
	Given $f: P^2_< \to Q$ and $[a,b) \in P^2_<$, for brevity we will write $f[a,b)$ for $f([a,b))$.
	Given an injective poset map $\iota:[m+1] \to \R$, there is a corresponding poset map $(\iota,\iota): \dgm \to \Dgm$.
	
	\subsection{Rank functions on half-open intervals}
	\label{sec:rank-half-open}
	
	Let $P$ be a sub-poset of $\R$ and
	let $M$ be a persistence module indexed by $P$.
	Recall from \cref{sec:pm-rank} that we have the rank function $\rank(M): \intc(P)^{\op} \to \Z_+$.
	Consider the half-open interval $[a,b) \in P^2_<$ and the function $\rank(M)[a,-]: [a,b)^{\op} \to \Z_+$.
	Define $\Rank(M):(P^2_<)^{\op} \to \Z_+$ by $\Rank(M)[a,b) = \lim_{c \to b^-} \rank(M)[a,c] = \min_{a \leq c < b} \rank(M)[a,c]$.
	
	\begin{example}
		Consider a persistence module $M$ indexed by $[m]$ and let $\hat{M}$ be the corresponding persistence module indexed by $[m+1]$.
		For $[i,j) \in \dgm$, $\Rank(\hat{M})[i,j) = \rank(\hat{M})[i,j-1] = \rank(M)[i,j-1]$.
		Thus we will sometimes write $\Rank(M)$ for $\Rank(\hat{M})$.
		That is, for a persistence module $M$ indexed by $[m]$ we have the poset map $\Rank(M): (\dgm)^{\op} \to \Z_+$, given by $\Rank(M)[a,b) = \rank(M)[a,b-1]$. For a persistence module $M$ indexed by $\R$, the equality $\Rank(M)[a,b) = \lim_{c \to b^-} \rank(M)[a,c] = \min_{a \leq c < b} \rank(M)[a,c]$ follows.
	\end{example}
	
	Consider a persistence module $M$ indexed by $[m]$ and an injective map $\iota: [m+1] \to \R$. Let $\hat{M}$ and $\overline{M}$ be the corresponding persistence modules indexed by $[m+1]$ and $\R$, respectively.
	
	\begin{lemma} \label{lem:Rank}
		Let $[a,b) \in \Dgm$.
		Then $\Rank(\overline{M})[a,b) = \Rank(\hat{M})[i,j) = \rank(M)[i,j-1]$, where $i$ is the largest element of $[m+1]$ such that $\iota(i) \leq a$ and $j$ is the smallest element of $[m+1]$ such that $\iota(j) \geq b$. The value of $\Rank(\overline{M})[a,b)$ is $0$ if there are no such elements.
	\end{lemma}
	
	\begin{proof}
		Observe that $\Rank(\overline{M})[a,b) = \lim_{c \to b^-} \rank(\overline{M})[a,c] = \rank(\hat{M})[i,k]$, where $i$ is the largest element of $[m+1]$ such that $\iota(i) \leq a$ and $k$ is the largest element of $[m+1]$ such that $\iota(k) < b$ and that $\Rank(\overline{M})[a,b)$ is $0$ if there are no such elements.
		In the first case, $\rank(\hat{M})[i,k]=\Rank(\hat{M})[i,j)$ where $j$ is the smallest element of $[m+1]$ such that $\iota(j) \geq b$. Note that there is no such element if $k = m+1$, but in this case $\rank(\hat{M})[i,k]=0$.
	\end{proof}
	
	\begin{figure}
		\begin{equation*}
			\begin{tikzcd}
				(\dgm)^{\op} \ar[rr,"\Rank(\hat{M})"] \ar[d,"{(\iota,\iota)}"',hook] 
				& & \Z_{+} \\
				(\Dgm)^{\op} \ar[urr,"\Rank(\overline{M})"']      
			\end{tikzcd}
		\end{equation*}
		\caption{The rank functions on half-open intervals associated to a persistence module indexed by $[m]$ and an injective map $\iota:[m+1] \to \R$. The rank function $\Rank(\overline{M})$ is a canonical extension of the rank function $\Rank(\hat{M})$.}
		\label{fig:rank}
	\end{figure}
	
	See Figure~\ref{fig:rank}. In categorical language, \cref{lem:Rank} says that $\Rank(\overline{M})$ is the left Kan extension of $\Rank(\hat{M})$ along $(\iota,\iota)$.

	\subsection{Discrete persistence diagrams} \label{sec:pd}
	
	In this section we show how persistence diagrams can be obtained from rank functions on half-open intervals using M\"obius inversion.
	
	The poset $\dgm$ may be visualized as a discrete grid of points
	in the plane
	(see \cref{fig:dgm_poset}). 
	Consider the incidence algebra on $\dgm$ with values in $\bZ$. Elements of the incidence algebra are functions $\intc(\dgm) \to \bZ$. The members of $\intc(\dgm)$ are intervals of the form $[[x,y),[x',y')]$ where
	$[x,y) \subset [x',y')$.
	Going forward,
	let $\mu,\delta,$ and $\zeta$ be the corresponding functions in the incidence algebra on $\dgm$ with values in $\bZ$ unless otherwise noted and let $M$ be a persistence module with indexing poset $[m]$. 
	
	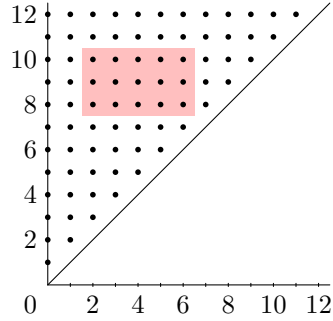
\begin{figure}
		\begin{center}
			\begin{tikzpicture}[line cap=round,line join=round,>=triangle 45,x=1.0cm,y=1.0cm,scale=0.3]
			\draw[color=black] (0,0) -- (12.5,0);
			\foreach \x in {1,2,...,12}
			\draw[shift={(\x,0)},color=black] (0pt,2pt) -- (0pt,-2pt);
			\foreach \x in {2,4,...,12}
			\draw[shift={(\x,0)},color=black] node[below] {\footnotesize $\x$};
			\draw[color=black] (0,0) -- (0,12.5);
			\foreach \y in {2,4,...,12}
			\draw[shift={(0,\y)},color=black] node[left] {\footnotesize $\y$};
			\draw[color=black] (0,0) node[below left] {\footnotesize $0$};
			\draw[color=black] (0,0) -- (12.5,12.5);
			
			\fill[fill=red!25] (6.5,7.5) rectangle (1.5,10.5);
			
			\foreach \x in {0,1,2,...,12} 
			\foreach \y in {0,1,2,...,12} 
			{
				\ifnum \x < \y
				\fill [color=black] (\x,\y) circle (1.25mm); 
				\fi
			}

			\end{tikzpicture}
		\end{center}
		\caption{Visualizing $\dgm$ for $m=11$. Black points $(x,y)$ correspond to half-open intervals $[x,y)$. The points in the shaded region are elements of the interval $[[6,8),[2,10)]$ in $\intc(\dgm)$.}
		\label{fig:dgm_poset} 
	\end{figure}

	\begin{proposition} \label{prop:mobius}
		The M\"obius function $\mu:\intc(\dgm) \to \bZ$ is given by
		$\mu([x,y),[x,y)) = \mu([x,y),[x-1,y+1)) = 1$,
		$\mu([x,y),[x-1,y)) = \mu([x,y),[x,y+1)) = -1$, and
		$\mu([x,y),J) = 0$ otherwise.
	\end{proposition} 
	
	\begin{proof} 
		From \cref{def:functions},
		\begin{gather*}
			\mu([x,y),[x,y)) = 1, \\
			\mu([x,y),[x-1,y))=
			-\mu([x,y),[x,y)) = -1, \\
			\mu([x,y),[x,y+1)) =
			-\mu([x,y),[x,y)) = -1, \text{ and} \\
			\mu([x,y),[x-1,y+1)) =
			-\mu([x,y),[x,y)) - \mu([x,y),[x-1,y)) - \mu([x,y),[x,y+1)) = 1\text{.}
		\end{gather*}
		We also have that
		\begin{equation*}
			\mu([x,y),[x,y+2)) = - \mu([x,y),[x,y)) - \mu([x,y),[x,y+1)) = 0,
		\end{equation*}
		
		\vspace{-2em}
		
		\begin{multline*}
			\mu([x,y),[x-1,y+2)) = - \mu([x,y),[x,y)) - \mu([x,y),[x,y+1)) - \mu([x,y),[x-1,y)) \\ - \mu([x,y),[x-1,y+1)) - \mu([x,y),[x,y+2)) = 0,
		\end{multline*}
		and similarly $\mu([x,y),[x-2,y))=0$ and $\mu([x,y),[x-2,y+1))=0$.
		By induction, $\mu([x,y),[x',y')) = 0$ in all other cases.
	\end{proof}
	
	Combining \eqref{eq:convolution} and \cref{prop:mobius}, we have the following.
	
	\begin{corollary}\label{cor:inversion}
		For any $h:\dgm \to \mz$,
		\begin{gather*}
			(\mu*h)[x,y) = h[x,y) - h[x-1,y) - h[x,y+1) + h[x-1,y+1) \text{ if } 1 \leq x < y \leq m, \\ 
			(\mu*h)[x,m+1) = h[x,m+1) - h[x-1,m+1) \text{ if } x \geq 1, \\
			(\mu*h)[0,y) = h[0,y) - h[0,y+1) \text{ if } y \leq m, \text{ and}\\
			(\mu*h)[0,m+1) = h[0,m+1). 
		\end{gather*}
	\end{corollary}

	\begin{definition} 
		The \emph{persistence diagram} of $M$ is the function $\pd:\dgm\to\mz$ given by $\pd := \mu * \Rank$, where $\Rank = \Rank(M)$. 
	\end{definition} 
	
	Persistence diagrams~\cite{cseh:stability} are one of the most popular summaries of persistence modules. Observe that $\zeta*\pd = \zeta*\mu*\Rank = \Rank$. It follows as a consequence of \cref{thm:gabriel} that a persistence module (indexed by $[m]$) is determined up to isomorphism by its persistence diagram.
	For an example of computing the persistence diagram from the rank function and then recovering the rank function from the persistence diagram, see Figures \ref{fig:rank_and_persistence_diagram},
	\ref{fig:PD_calculation}, and
	\ref{fig:zeta_convolution}.
	
	\begin{figure}
		\begin{center}
			\begin{tikzpicture}[line cap=round,line join=round,>=triangle 45,x=1.0cm,y=1.0cm,scale=0.3]
			\draw[color=black] (0,0) -- (12.5,0);
			\foreach \x in {1,2,...,12}
			\draw[shift={(\x,0)},color=black] (0pt,2pt) -- (0pt,-2pt);
			\foreach \x in {2,4,...,12}
			\draw[shift={(\x,0)},color=black] node[below] {\footnotesize $\x$};
			\draw[color=black] (0,0) -- (0,12.5);
			\foreach \y in {2,4,...,12}
			\draw[shift={(0,\y)},color=black] node[left] {\footnotesize $\y$};
			\draw[color=black] (0,0) node[below left] {\footnotesize $0$};
			\draw[color=black] (0,0) -- (12.5,12.5);
			
			\foreach \x in {0,1} 
			\foreach \y in {0,1,2,...,12} 
			{
				\ifnum \x < \y
				\draw [color=black] (\x,\y) circle (1.25mm); 
				\fi
			}
			
			\foreach \x in {2,3} 
			\foreach \y in {9,10,...,12} 
			{
				\ifnum \x < \y
				\draw [color=black] (\x,\y) circle (1.25mm); 
				\fi
			}	
			\foreach \x in {2,3} 
			\foreach \y in {2,3,...,8} 
			{
				\ifnum \x < \y
				\draw[shift={(\x,\y)}] node {$1$}; 
				\fi
			}
			
			\foreach \x in {4,5} 
			\foreach \y in {9,10,...,12} 
			{
				\ifnum \x < \y
				\draw[shift={(\x,\y)}] node {$1$}; 
				\fi
			}	
			
			\foreach \x in {6,7,...,12} 
			\foreach \y in {11,12} 
			{
				\ifnum \x < \y
				\draw[shift={(\x,\y)}] node {$1$}; coordinates{(\x,\y)}; 
				\fi
			}
			
			\foreach \x in {4,5} 
			\foreach \y in {4,5,...,8} 
			{
				\ifnum \x < \y
				\draw[shift={(\x,\y)}] node {$2$}; 
				\fi
			}	
			
			\foreach \x in {6,7,...,10} 
			\foreach \y in {9,10} 
			{
				\ifnum \x < \y
				\draw[shift={(\x,\y)}] node {$2$};  
				\fi
			}	
			
			\foreach \x in {6,7,8} 
			\foreach \y in {6,7,8} 
			{
				\ifnum \x < \y
				\draw[shift={(\x,\y)}] node {$3$}; 
				\fi
			}		
			
			\end{tikzpicture}
			\begin{tikzpicture}[line cap=round,line join=round,>=triangle 45,x=1.0cm,y=1.0cm,scale=0.3]
			\draw[color=black] (0,0) -- (12.5,0);
			\foreach \x in {1,2,...,12}
			\draw[shift={(\x,0)},color=black] (0pt,2pt) -- (0pt,-2pt);
			\foreach \x in {2,4,...,12}
			\draw[shift={(\x,0)},color=black] node[below] {\footnotesize $\x$};
			\draw[color=black] (0,0) -- (0,12.5);
			\foreach \y in {2,4,...,12}
			\draw[shift={(0,\y)},color=black] node[left] {\footnotesize $\y$};
			\draw[color=black] (0,0) node[below left] {\footnotesize $0$};
			\draw[color=black] (0,0) -- (12.5,12.5);
			
			\foreach \x in {0,1,2,...,12} 
			\foreach \y in {0,1,2,...,12} 
			{
				\ifnum \x < \y
				\draw [color=black] (\x,\y) circle (1.25mm); 
				\fi
			}
			
			\fill [color=blue] (6,10) circle (3mm);
			\fill [color=blue] (4,12) circle (3mm);
			\fill [color=blue] (2,8) circle (3mm);
			\end{tikzpicture}
			
		\end{center}
		\caption{Visualizations of the functions $\Rank,\pd:\dgm\to\mz$ for a persistence module. {\bf Left}: The values of $\Rank$ in the plane. Circles indicate intervals that evaluate to $0$. {\bf Right}: Dark blue disks indicate elements of $\dgm$ where $\pd$ evaluates to $1$.  Other intervals evaluate to $0$.  For an example of the computation of $\pd$ see Figure~\ref{fig:PD_calculation}.}
		\label{fig:rank_and_persistence_diagram}
	\end{figure}
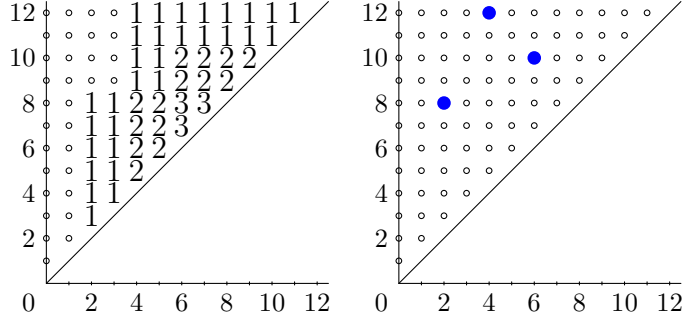
	
	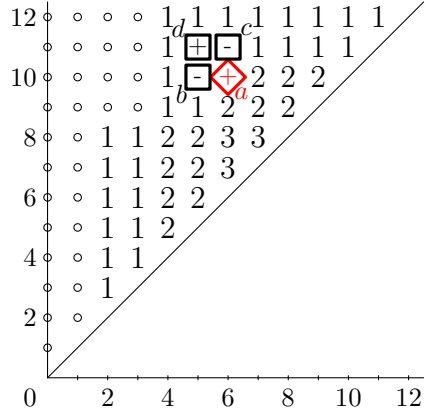
\begin{figure}
		\begin{center}
			\begin{tikzpicture}[line cap=round,line join=round,>=triangle 45,x=1.0cm,y=1.0cm,scale=0.4]
			
			\draw[color=black] (0,0) -- (12.5,0);
			\foreach \x in {1,2,...,12}
			\draw[shift={(\x,0)},color=black] (0pt,2pt) -- (0pt,-2pt);
			\foreach \x in {2,4,...,12}
			\draw[shift={(\x,0)},color=black] node[below] {\footnotesize $\x$};
			\draw[color=black] (0,0) -- (0,12.5);
			\foreach \y in {2,4,...,12}
			\draw[shift={(0,\y)},color=black] node[left] {\footnotesize $\y$};
			\draw[color=black] (0,0) node[below left] {\footnotesize $0$};
			\draw[color=black] (0,0) -- (12.5,12.5);
			
			\foreach \x in {0,1} 
			\foreach \y in {0,1,2,...,12} 
			{
				\ifnum \x < \y
				\draw [color=black] (\x,\y) circle (1.25mm); 
				\fi
			}
			
			\foreach \x in {2,3} 
			\foreach \y in {9,10,...,12} 
			{
				\ifnum \x < \y
				\draw [color=black] (\x,\y) circle (1.25mm); 
				\fi
			}	
			\foreach \x in {2,3} 
			\foreach \y in {2,3,...,8} 
			{
				\ifnum \x < \y
				\draw[shift={(\x,\y)}] node {$1$}; 
				\fi
			}
			
			\foreach \x in {4,5} 
			\foreach \y in {9,12} 
			{
				\ifnum \x < \y
				\draw[shift={(\x,\y)}] node {$1$}; 
				\fi
			}	
			
			\foreach \x in {4} 
			\foreach \y in {10,11} 
			{
				\ifnum \x < \y
				\draw[shift={(\x,\y)}] node {$1$}; 
				\fi
			}	
			
			\foreach \x in {6,7,...,12} 
			\foreach \y in {12} 
			{
				\ifnum \x < \y
				\draw[shift={(\x,\y)}] node {$1$}; 
				\fi
			}
			
			\foreach \x in {7,...,12} 
			\foreach \y in {11} 
			{
				\ifnum \x < \y
				\draw[shift={(\x,\y)}] node {$1$};
				\fi
			}	
			
			\foreach \x in {4,5} 
			\foreach \y in {4,5,...,8} 
			{
				\ifnum \x < \y
				\draw[shift={(\x,\y)}] node {$2$}; 
				\fi
			}	
			
			\foreach \x in {7,...,10} 
			\foreach \y in {9,10} 
			{
				\ifnum \x < \y
				\draw[shift={(\x,\y)}] node {$2$}; 
				\fi
			}	
			
			\foreach \x in {6} 
			\foreach \y in {9} 
			{
				\ifnum \x < \y
				\draw[shift={(\x,\y)}] node {$2$}; 
				\fi
			}	
			\foreach \x in {6,7,8} 
			\foreach \y in {6,7,8} 
			{
				\ifnum \x < \y
				\draw[shift={(\x,\y)}] node {$3$}; 
				\fi
			}				
			
			\draw (6,10) node[diamond,
			color=red,draw,inner sep=0pt,scale=0.8,very thick] {+};
			\draw (5,10) node[color=black,draw,inner sep=0pt,scale=0.8,very thick,minimum size=4mm] {-};
			\draw (6,11) node[color=black,draw,inner sep=0pt,scale=0.8,very thick,minimum size=4mm] {-};
			\draw (5,11) node[color=black,draw,scale=0.8,inner sep=0pt,very thick,minimum size=4mm] {+};
			
			\draw[color=red] (6.45,9.45) node {\footnotesize ${a}$};
			\draw[color=black] (4.45,9.45) node {\footnotesize ${b}$};
			\draw[color=black] (6.6,11.6) node {\footnotesize ${c}$};
			\draw[color=black] (4.35,11.65) node {\footnotesize ${d}$};
			\end{tikzpicture} 
		\end{center}
		\caption{
			Calculating $\pd$, which is $\mu*\Rank$ by definition, from the $\Rank$ function using \cref{cor:inversion}. The values of $\Rank$ (depicted as in \cref{fig:rank_and_persistence_diagram}) are summed with the indicated sign to obtain the value of $\pd$ at the bottom right point.
			$\pd({a}) = \Rank({a}) - \Rank({b}) - \Rank({c}) + \Rank({d}) = 2 - 1 - 1 + 1 = 1$.
		}\label{fig:PD_calculation}
	\end{figure} 
	
	\begin{figure}
		\begin{center}
			\begin{tikzpicture}[line cap=round,line join=round,>=triangle 45,x=1.0cm,y=1.0cm,scale=0.3]
			\fill[fill=red!25] (6.5,7.5) rectangle (0,12.5);
			\draw[color=black] (0,0) -- (12.5,0);
			\foreach \x in {1,2,...,12}
			\draw[shift={(\x,0)},color=black] (0pt,2pt) -- (0pt,-2pt);
			\foreach \x in {2,4,...,12}
			\draw[shift={(\x,0)},color=black] node[below] {\footnotesize $\x$};
			\draw[color=black] (0,0) -- (0,12.5);
			\foreach \y in {2,4,...,12}
			\draw[shift={(0,\y)},color=black] node[left] {\footnotesize $\y$};
			\draw[color=black] (0,0) node[below left] {\footnotesize $0$};
			\draw[color=black] (0,0) -- (12.5,12.5);
			
			\draw [black] plot [only marks, mark=star,mark size=2.5mm] coordinates{(6,8)};
			
			\foreach \x in {0,1,2,...,5,7,8,...,12} 
			\foreach \y in {0,1,2,...,12} 
			{
				\ifnum \x < \y
				\draw [color=black] (\x,\y) circle (1.25mm); 
				\fi
			}
			\foreach \y in {0,...,7,9,10,11,12}
			{
				\ifnum 6 < \y
				\draw [color=black] (6,\y) circle (1.25mm); 
				\fi
			}
			
			\fill [color=blue] (6,10) circle (3mm);
			\fill [color=blue] (4,12) circle (3mm);
			\fill [color=blue] (2,8) circle (3mm);
			\end{tikzpicture}
			
		\end{center}
		\caption{
			Calculation of $\Rank$ from $\pd:\dgm\to\mz$ via the equality $\Rank = \zeta * \pd$. Dark blue disks indicate elements of $\dgm$ where $\pd$ evaluates to $1$. Other elements evaluate to $0$. The value $\Rank[6,8) = 3$, indicated by a star, is obtained by summing all values of $\pd$ up and to the left of $[6,8)$ in the highlighted region. 
		}\label{fig:zeta_convolution}
	\end{figure}
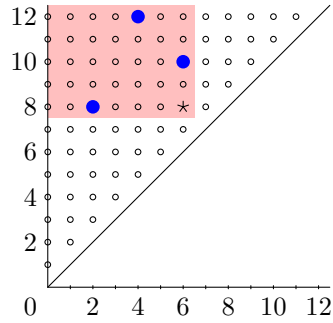

	\subsection{Continuous persistence diagrams} \label{sec:cont}
	
	Let $M$ be a persistence module index by $[m]$ with $\overline{M}$ the corresponding persistence indexed by $\R$ via $\iota:[m+1]\to\R$.
	
	\begin{definition}
		\label{def:continuous_diagrams} The persistence diagram of $\overline{M}$ is the function (not necessarily order-reversing) $\pd(\overline{M}):\R^2_<\to\mz$ given by 
		\[
		\pd(\overline{M})[a,b) = \begin{cases} 
		\pd[x,y) & \text{if }[a,b) = [\iota(x),\iota(y)) \\ 
		0 & \text{otherwise.}
		\end{cases}
		\]
	\end{definition}
	
	\begin{remark}
		Note that $\zeta*\pd(\overline{M})$ is well-defined when $\overline{M}$ is of the form given. One may check that $\Rank(\overline{M}) = \zeta*\pd(\overline{M})$.
	\end{remark}

	\section{Graded rank function and graded persistence diagrams} \label{sec:gpd}
	
	In this section we introduce graded versions of the $\rank$ function and persistence diagrams. \Cref{thm:consistency} establishes the relationship between these graded functions and their ungraded counterparts.
	
	\subsection{Graded rank function}
	\label{sec:graded-rank}
	
	Using unary numbers we obtain a graded version of the rank function.
	
	\begin{definition} \label{def:uk}
		For any natural number $k\geq 1$, let $u_k:\mz_{\geq 0} \to \mz$ be the step function given by 
		\[
		u_k(n) = 
		\begin{cases}
		1 & \text{if } n \geq k \\ 
		0 & \text{otherwise.} 
		\end{cases}
		\]
		Given any $n\in\mz_{\geq 0}$ we can form the sequence $(u_k(n))_{k\geq 1}$. For example, if $n=5$ we obtain the sequence $(1,1,1,1,1,0,0,\dots)$. This sequence is called the \emph{unary representation} of $n$ and its sum is $n$.
		Colloquially, it represents a number using \emph{tally marks}.
		More abstractly, we have the function $(u_k)_{k \geq 1}: \mz_{\geq 0} \to \bigoplus_{k\geq 1} \mz$.
		Let $\ptwise$ denote the function $\oplus_{k \geq 1} \mz \to \mz$ given by summation. This function is well defined since the sequence $(a_k)_{k \geq 1}$ in $\bigoplus_{k\geq 1} \mz$ has only finitely many nonzero terms.
		Furthermore,
		for all $n \in \mz_{\geq 0}$, $\ptwise (u_k)_{k\geq 1} (n) = n$.
	\end{definition}

	Recall from \cref{sec:rank-half-open} that for a persistence module $M$ indexed by $P$ we have a corresponding poset map $\Rank(M): (P^2_<)^{\op} \to \Z_+$.
	Note that $u_k$ is a poset map $u_k: \Z_+ \to \Z_+$.
	
	\begin{definition} 
		\label{def:graded-rank}
		The $k$-th \emph{graded rank function} of $M$ is the poset map $\Rank_k(M): (P^2_<)^{\op} \to \Z_+$ defined by $\Rank_k(M) = u_k \circ \Rank(M)$.
		The \emph{graded rank function} $\Rank_*(M): P^2_< \to \bigoplus_{k\geq 1} \mz$ is
		given by $\Rank_* = (\Rank_k)_{k\geq 1}$.
	\end{definition}

	\begin{figure}
		\begin{equation*}
			\begin{tikzcd}
				(\dgm)^{\op} \ar[rr,"\Rank(\hat{M})"] \ar[d,"{(\iota,\iota)}"',hook] 
				& & \Z_{+} \ar[rr,"u_k"] & & \Z_{+} \\
				(\Dgm)^{\op} \ar[urr,"\Rank(\overline{M})"']      
			\end{tikzcd}
		\end{equation*}
		\caption{The graded rank functions on half-open intervals associated to a persistence module indexed by $[m]$ and an injective map $\iota:[m+1] \to \R$. The graded rank function $\Rank_k(\overline{M})$ given by $u_k \circ \Rank(\overline{M})$ is a canonical extension of the rank function $\Rank_k(\hat{M})$ given by $u_k \circ \Rank(\hat{M})$.}
		\label{fig:graded-rank}
	\end{figure}
	
	See \cref{fig:graded-rank}. Consider a persistence module $M$ indexed by $[m]$ and injective map $\iota:[m+1] \to \R$.
	We have corresponding persistence modules $\hat{M}$ indexed by $[m+1]$ and $\overline{M}$ indexed by $\R$.
	By \cref{def:graded-rank}, we have $\Rank_k(\hat{M}): \dgm \to \Z_+$ and $\Rank_k(\overline{M}): \Dgm \to \Z_+$.
	Recall that we sometimes write $\Rank(M)$ for $\Rank(\hat{M})$.
	In categorical language, $\Rank_k(\overline{M})$ is the left Kan extension of $\Rank_k(\hat{M})$ along $(\iota,\iota)$.
	We may also say that the support of $\Rank_k(\overline{M})$ is the downward closure of the support of $\Rank_k(\hat{M})$.
	
	\subsection{Graded persistence diagram}
	\label{sec:graded-pd}
	
	Applying M\"obius inversion to the graded rank function, we obtain the graded persistence diagram.
	
	\begin{definition}
		The $k$-th \emph{graded persistence diagram} of $M$ is the function $\pd_k(M): \dgm \to \mz$ given by
		$\pd_k(M) = \mu * \Rank_k(M)$ and
		the \emph{graded persistence diagram} is the function
		$\pd_*(M): \dgm \to \bigoplus_{k\geq 1} \mz$ given by
		$\pd_*(M) = (\pd_k(M))_{k\geq 1}$. 
	\end{definition}
	
	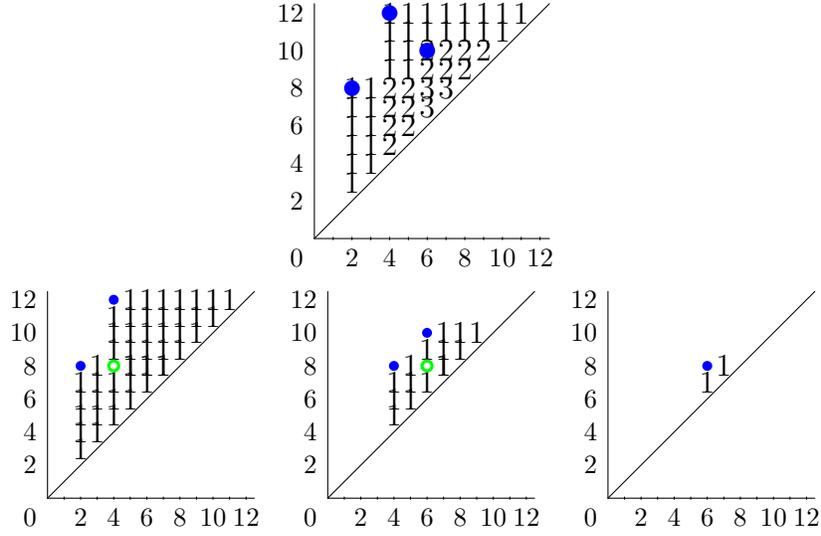
\begin{figure}
		\begin{center} 
			\begin{tikzpicture}[line cap=round,line join=round,>=triangle 45,x=1.0cm,y=1.0cm,scale=0.25]
			\draw[color=black] (0,0) -- (12.5,0);
			\foreach \x in {1,2,...,12}
			\draw[shift={(\x,0)},color=black] (0pt,2pt) -- (0pt,-2pt);
			\foreach \x in {2,4,...,12}
			\draw[shift={(\x,0)},color=black] node[below] {\footnotesize $\x$};
			\draw[color=black] (0,0) -- (0,12.5);
			\foreach \y in {2,4,...,12}
			\draw[shift={(0,\y)},color=black] node[left] {\footnotesize $\y$};
			\draw[color=black] (0,0) node[below left] {\footnotesize $0$};
			\draw[color=black] (0,0) -- (12.5,12.5);
			
			\foreach \x in {2,3} 
			\foreach \y in {2,3,...,8} 
			{
				\ifnum \x < \y
				\draw[shift={(\x,\y)}] node {$1$}; 
				\fi
			}
			
			\foreach \x in {4,5} 
			\foreach \y in {9,10,...,12} 
			{
				\ifnum \x < \y
				\draw[shift={(\x,\y)}] node {$1$}; 
				\fi
			}	
			
			\foreach \x in {6,7,...,12} 
			\foreach \y in {11,12} 
			{
				\ifnum \x < \y
				\draw[shift={(\x,\y)}] node {$1$}; coordinates{(\x,\y)}; 
				\fi
			}
			
			\foreach \x in {4,5} 
			\foreach \y in {4,5,...,8} 
			{
				\ifnum \x < \y
				\draw[shift={(\x,\y)}] node {$2$}; 
				\fi
			}	
			
			\foreach \x in {6,7,...,10} 
			\foreach \y in {9,10} 
			{
				\ifnum \x < \y
				\draw[shift={(\x,\y)}] node {$2$};  
				\fi
			}	
			
			\foreach \x in {6,7,8} 
			\foreach \y in {6,7,8} 
			{
				\ifnum \x < \y
				\draw[shift={(\x,\y)}] node {$3$}; 
				\fi
			}		
			
			\fill [color=blue] (6,10) circle (4.2mm);
			\fill [color=blue] (4,12) circle (4.2mm);
			\fill [color=blue] (2,8) circle (4.2mm);
			\end{tikzpicture}
			
			\begin{tikzpicture}[line cap=round,line join=round,>=triangle 45,x=1.0cm,y=1.0cm,scale=0.22]
			\draw[color=black] (0,0) -- (12.5,0);
			\foreach \x in {1,2,...,12}
			\draw[shift={(\x,0)},color=black] (0pt,2pt) -- (0pt,-2pt);
			\foreach \x in {2,4,...,12}
			\draw[shift={(\x,0)},color=black] node[below] {\footnotesize $\x$};
			\draw[color=black] (0,0) -- (0,12.5);
			\foreach \y in {2,4,...,12}
			\draw[shift={(0,\y)},color=black] node[left] {\footnotesize $\y$};
			\draw[color=black] (0,0) node[below left] {\footnotesize $0$};
			\draw[color=black] (0,0) -- (12.5,12.5);
			\foreach \x in {2,3,...,8} 
			\foreach \y in {2,3,...,8} 
			{			
				\ifnum \x < \y
				\ifnum \x = 2
				\ifnum \y = 8
				
				\else 
				\draw[shift={(\x,\y)}] node {$1$};
				\fi
				\else
				\ifnum \x = 4
				\ifnum \y=8
				\else
				\draw[shift={(\x,\y)}] node {$1$};
				\fi 
				\else 
				\draw[shift={(\x,\y)}] node {$1$};
				\fi
				\fi 
				\fi
			}	
			\foreach \x in {4,5,...,12} 
			\foreach \y in {8,9,...,12} 
			{				
				\ifnum \x < \y
				\ifnum \x = 4
				\ifnum \y = 8
				
				\else 
				\ifnum \y = 12
				\else
				\draw[shift={(\x,\y)}] node {$1$};
				\fi
				\fi
				\else
				\draw[shift={(\x,\y)}] node {$1$};
				\fi
				\fi
			}							
			
			\fill [color=blue] (4,12) circle (3mm);
			\fill [color=blue] (2,8) circle (3mm);
			\fill [color=white] (4,8) circle (3mm);
			\draw [very thick,color=green] (4,8) circle (3mm);
			\end{tikzpicture}
			\begin{tikzpicture}[line cap=round,line join=round,>=triangle 45,x=1.0cm,y=1.0cm,scale=0.22]
			\draw[color=black] (0,0) -- (12.5,0);
			\foreach \x in {1,2,...,12}
			\draw[shift={(\x,0)},color=black] (0pt,2pt) -- (0pt,-2pt);
			\foreach \x in {2,4,...,12}
			\draw[shift={(\x,0)},color=black] node[below] {\footnotesize $\x$};
			\draw[color=black] (0,0) -- (0,12.5);
			\foreach \y in {2,4,...,12}
			\draw[shift={(0,\y)},color=black] node[left] {\footnotesize $\y$};
			\draw[color=black] (0,0) node[below left] {\footnotesize $0$};
			\draw[color=black] (0,0) -- (12.5,12.5);
			
			\foreach \x in {0,1,2,3} 
			\foreach \y in {0,1,2,...,12} 
			{
				\ifnum \x < \y
				\fi
			}
			
			\foreach \x in {4,5,...,7} 
			\foreach \y in {9,10,...,12} 
			{
				\ifnum \x < \y
				\fi
			}
			
			\foreach \x in {6,7,...,12} 
			\foreach \y in {11,12} 
			{
				\ifnum \x < \y 
				\fi
			}
			
			\foreach \x in {4,5,...,10} 
			\foreach \y in {4,5,...,8} 
			{		
				\ifnum \x < \y
				\ifnum \x = 4
				\ifnum \y = 8
				
				\else 
				\draw[shift={(\x,\y)}] node {$1$};
				\fi
				\else
				\ifnum \x = 6
				\ifnum \y=8
				\else
				\draw[shift={(\x,\y)}] node {$1$};
				\fi 
				\else 
				\draw[shift={(\x,\y)}] node {$1$};
				\fi
				\fi 
				\fi
			}	
			\foreach \x in {6,7,...,10} 
			\foreach \y in {8,9,10} 
			{
				\ifnum \x < \y
				\ifnum \x = 6
				\ifnum \y = 10
				
				\else
				\ifnum \y = 8
				
				\else
				\draw[shift={(\x,\y)}] node {$1$};
				\fi
				\fi
				
				\else
				\draw[shift={(\x,\y)}] node {$1$};
				\fi
				\fi
			}

			\fill [color=blue] (6,10) circle (3mm);
			\fill [color=white] (6,8) circle (3mm);
			\draw [very thick,green=blue] (6,8) circle (3mm);
			\fill [color=blue] (4,8) circle (3mm);
			\end{tikzpicture}
			\begin{tikzpicture}[line cap=round,line join=round,>=triangle 45,x=1.0cm,y=1.0cm,scale=0.22]
			\draw[color=black] (0,0) -- (12.5,0);
			\foreach \x in {1,2,...,12}
			\draw[shift={(\x,0)},color=black] (0pt,2pt) -- (0pt,-2pt);
			\foreach \x in {2,4,...,12}
			\draw[shift={(\x,0)},color=black] node[below] {\footnotesize $\x$};
			\draw[color=black] (0,0) -- (0,12.5);
			\foreach \y in {2,4,...,12}
			\draw[shift={(0,\y)},color=black] node[left] {\footnotesize $\y$};
			
			\draw[color=black] (0,0) node[below left] {\footnotesize $0$};
			\draw[color=black] (0,0) -- (12.5,12.5);			
			\draw[shift={(6,7)}] node {$1$};
			\draw[shift={(7,8)}] node {$1$};

			\fill [color=blue] (6,8) circle (3mm);
			\end{tikzpicture}
		\end{center}
		\caption{ The functions $\Rank$, $\pd$, $\Rank_k$, and $\pd_k$ for a persistence module. {\bf Top}: The $\Rank$ function and $\pd$. Dark blue disks are where $\pd$ evaluates to $1$. {\bf Left-to-right on bottom row}: $\Rank_k$ and $\pd_k$ for $k=1,2,3$. Dark blue disks are where $\pd_k$ evaluates to $1$, and light green circles are where $\pd_k$ evaluates to $-1$. Hidden numbers underneath the disks and circles are all $1$ except for the rightmost disk in the top figure where the hidden value is $2$. }\label{fig:graded_functions}
	\end{figure}
	
	For simplicity, we omit $M$ from the notation when there is no risk of confusion. The graded persistence diagram of $\overline{M}$ (where $\overline{M}:\R\to\vect_K$ is defined as in \cref{sec:discrete-cont}) is defined in the same way as the persistence diagram of $\overline{M}$ in \cref{def:continuous_diagrams}.
	
	\begin{proposition} \label{prop:commute} 
		Consider $f_k:\dgm \to \mz$ for $k \geq 1$ such that for $[a,b) \in \dgm$, $f_k[a,b)=0$ for all but finitely many $k$.
		Then $\ptwise(\mu * f_k)_{k \geq 1} = \mu * (\ptwise (f_k)_{k \geq 1})$.
	\end{proposition}
	
	\begin{proof}
		Let $I$ be an interval in $\dgm$. 
		\begin{align*}
			\ptwise(\mu*f_k)_{k\geq 1}(I)
			& = \sum_{k\geq 1} \sum_{I \subset I'} \mu(I,I')f_k(I') \\
			& = \sum_{I \subset I'} \mu(I,I')\sum_{k\geq 1} f_k(I') \\
			& = \sum_{I \subset I'} \mu(I,I') \ptwise (f_k)_{k \geq 1}(I') \\
			& = \mu * (\ptwise(f_k)_{k\geq 1})(I) \qedhere
		\end{align*}
	\end{proof}
	
	\begin{theorem}[Consistency Theorem] \label{thm:consistency} 
		The following diagram commutes. That is, horizontal pairs of maps are inverses, $\ptwise \Rank_* = \Rank$ and $\ptwise \pd_* = \pd$.
		\par
		\adjustbox{scale=1,center}{
			\begin{tikzcd}[ampersand replacement=\&]
				\Rank \arrow[shift left=1ex,mapsto]{rr}{\mu * -} \arrow[shift right=2ex,swap,mapsto]{ddd}{(u_k\circ -)_{k\geq 1}}  \& \& \arrow[mapsto,dashed]{ll}{\zeta * -}  \pd  \\ \\ \\ 
				\arrow[shift right=2ex, swap,mapsto,dashed]{uuu}{\ptwise \circ -} \Rank_* \arrow[shift left=1ex,mapsto]{rr}{(\mu * -)_{k\geq 1}}  \& \& \arrow[shift right=1ex, swap,mapsto,dashed]{uuu}{\ptwise \circ -} \arrow[mapsto,dashed]{ll}{(\zeta * -)_{k\geq 1} } \pd_*
			\end{tikzcd} 
		}
		
	\end{theorem}
	\begin{proof} 
		We are given $\Rank: \dgm \to \mz_{\geq 0} \subset \mz$. $\pd$, $\Rank_*$, and $\pd_*$ are given by definition by the solid arrows:
		$\pd = \mu * \Rank$, $\Rank_k = u_k\circ\Rank$,
		$\Rank_* = (\Rank_k)_{k\geq 1}$,
		$\pd_k = \mu * \Rank_k$,
		and $\pd_* = (\pd_k)_{k\geq 1}$.
		Now consider the dashed arrows.
		The horizontal maps are inverses since $\zeta$ is the inverse of $\mu$ in the incidence algebra (\cref{prop:identities}).
		Recall that the composition $\ptwise (u_k)_{k\geq 1}$ is the identity on $\mz_{\geq 0}$. It follows that $\ptwise \Rank_* = \Rank$.
		Finally, by \cref{prop:commute}, we have that $\ptwise ( \mu * \Rank_k)_{k\geq 1} = \mu * (\ptwise \Rank_*) = \mu * \Rank$. That is, $\ptwise \pd_* = \pd$.
	\end{proof}

	\subsection{Support of the graded rank function}
	\label{sec:support}
	
	We now relate the $k$-th graded persistence diagram to the maximal elements of the support of the $k$-th graded rank function.
	
	The \emph{support} of a function $f:X \to \mz$ is the set $\{x \in X \ | \ f(x) \neq 0\}$.
	Since $\Rank_k$ is an order-reversing function from $[m+1]^2_<$ to $(\{0,1\},\leq)$ it follows that its support is a \emph{down-set}. That is, if $[x,y)$ is in the support of $\Rank_k$ and $[x',y') \subset [x,y)$ holds in $[m+1]^2_<$ then $[x',y')$ is in the support of $\Rank_k$. 
	The same is true for $\Rank_k$ as an order-reversing function from $\R^2_<$ to $(\{0,1\},\leq)$. Consider $\Rank_k: [m+1]^2_< \to \mz$. Recall that $\pd_k = \mu * \Rank_k$.
	
	\begin{proposition} \label{prop:pdk}
		Consider $[a,b) \in [m+1]^2_<$. Then $\pd_k[a,b) = 1$ if and only if $[a,b)$ is a maximal element in $[m+1]^2_<$ of the support of $\Rank_k$. Also, $\pd_k[a,b) = -1$ if and only if $[a,b)$ is the greatest lower bound of two maximal elements in $[m+1]^2_<$ of the support of $\Rank_k$.
	\end{proposition}
	
	\begin{proof}
		Recall that $\pd_k[a,b) = \Rank_k[a,b) - \Rank_k[a-1,b) - \Rank_k[a,b+1) + \Rank_k[a-1,b+1)$. Since $\Rank_k$ is order-reversing, $\pd_k[a,b) = 1$ if and only if $\Rank_k[a,b)=1$ and $\Rank_k[a-1,b) = \Rank_k[a,b+1) = \Rank_k[a-1,b+1) = 0$.
		Similarly, $\pd_k[a,b)=-1$ if and only if $\Rank_k[a,b) = \Rank_k[a-1,b) = \Rank_k[a,b+1) = 1$ and $\Rank_k[a-1,b+1) = 0$.
		In the first case, $[a,b)$ is a maximal element of the support of $\Rank_k$. 
		In the second case, since $\Rank_k$ is order-reversing, $\Rank_k[a-i,b+1) = \Rank_k[a-1,b+i) = 0$ for all $i\geq 1$. Since the support of $\Rank_k$ is finite, it follows that $[a,b)$ is the greatest lower bound (i.e. meet) of two maximal elements in the support of $\Rank_k$.
	\end{proof}
	
	It follows from this proposition that we can write the $k$-th graded persistence diagram, $\pd_k$, as a signed sum of indicator functions on half-open intervals (\cref{fig:gpd}). To simplify the notation we will denote the indicator function on a half-open interval by the corresponding half-open interval.
	Consider the graph with vertices  the half-open intervals in the support of $\pd_k$ and with edges between greatest lower bounds and maximal elements in $[m+1]^2_<$ of the support of $\Rank_k$.
	Then the support of $\pd_k$ can be partitioned according the connected components of this graph. Let $\ell$ denote the number of connected components. Order the connected components using the minimum of the coordinates of the vertices of each component. Let $m_i$ denote the number vertices in the $i$-th component which are maximal elements in the support of $\Rank_k$. Then there are $m_i - 1$ vertices in the $i$-th component which are greatest lower bounds of the $m_i$ maximal elements, for a total of $2m_i - 1$ vertices in the component.
	
	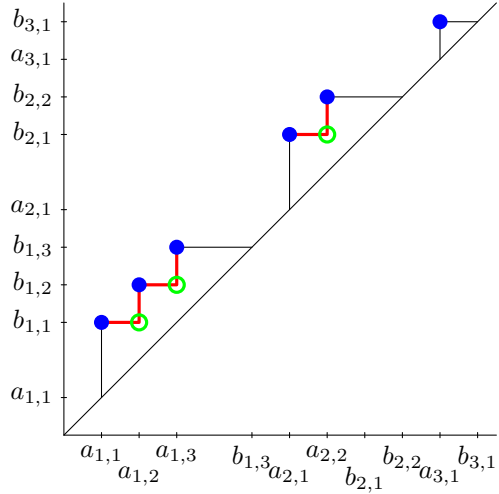
\begin{figure}
		\begin{center}
			\begin{tikzpicture}[line cap=round,line join=round,>=triangle 45,x=1.0cm,y=1.0cm,scale=0.5]
			\draw[color=black] (0,0) -- (11.5,0);
			\draw[color=black] (0,0) -- (0,11.5);
			\draw[color=black] (0,0) -- (11.5,11.5);
			
			\foreach \x in {1,2,3,5,6,7,8,9,10,11}
			\draw[shift={(\x,0)}] (0pt,2pt) -- (0pt,-2pt);
			\draw[shift={(1,0)}] node[below] {\footnotesize $a_{1,1}$};
			\draw[shift={(2,0)}] node[below,inner sep=1em] {\footnotesize $a_{1,2}$};
			\draw[shift={(3,0)}] node[below] {\footnotesize $a_{1,3}$};
			\draw[shift={(5,0)}] node[below] {\footnotesize $b_{1,3}$};
			\draw[shift={(6,0)}] node[below,inner sep=1em] {\footnotesize $a_{2,1}$};
			\draw[shift={(7,0)}] node[below] {\footnotesize $a_{2,2}$};
			\draw[shift={(8,0)}] node[below,inner sep=1em] {\footnotesize $b_{2,1}$};
			\draw[shift={(9,0)}] node[below] {\footnotesize $b_{2,2}$};
			\draw[shift={(10,0)}] node[below,inner sep=1em] {\footnotesize $a_{3,1}$};
			\draw[shift={(11,0)}] node[below] {\footnotesize $b_{3,1}$};
			
			\foreach \y in {1,3,4,5,6,8,9,10,11}
			\draw[shift={(0,\y)},color=black] (2pt,0pt) -- (-2pt,0pt);
			\draw[shift={(0,1)}] node[left] {\footnotesize $a_{1,1}$};
			\draw[shift={(0,3)}] node[left] {\footnotesize $b_{1,1}$};
			\draw[shift={(0,4)}] node[left] {\footnotesize $b_{1,2}$};
			\draw[shift={(0,5)}] node[left] {\footnotesize $b_{1,3}$};
			\draw[shift={(0,6)}] node[left] {\footnotesize $a_{2,1}$};
			\draw[shift={(0,8)}] node[left] {\footnotesize $b_{2,1}$};
			\draw[shift={(0,9)}] node[left] {\footnotesize $b_{2,2}$};
			\draw[shift={(0,10)}] node[left] {\footnotesize $a_{3,1}$};
			\draw[shift={(0,11)}] node[left] {\footnotesize $b_{3,1}$};
			
			\draw (1,1) -- (1,3) -- (2,3) -- (2,4) -- (3,4) -- (3,5) -- (5,5);
			\draw (6,6) -- (6,8) -- (7,8) -- (7,9) -- (9,9);
			\draw (10,10) -- (10,11) -- (11,11);
			\draw [very thick,red] (1,3) -- (2,3) -- (2,4) -- (3,4) -- (3,5);
			\draw [very thick,red] (6,8) -- (7,8) -- (7,9);
			
			\fill [color=blue] (1,3) circle (2mm);
			\draw [very thick,green=blue] (2,3) circle (2mm);
			\fill [color=blue] (2,4) circle (2mm);
			\draw [very thick,green=blue] (3,4) circle (2mm);
			\fill [color=blue] (3,5) circle (2mm);
			\fill [color=blue] (6,8) circle (2mm);
			\draw [very thick,green=blue] (7,8) circle (2mm);
			\fill [color=blue] (7,9) circle (2mm);
			\fill [color=blue] (10,11) circle (2mm);
			\end{tikzpicture}
		\end{center}
		\caption{An example of a $k$-th graded persistence diagram, $\pd_k$. Dark blue disks indicate where $\pd_k$ evaluates to 1 and light green circles indicate where $\pd_k$ evaluates to -1. Thick red lines are the edges of a graph whose vertices are the support of $\pd_k$. Vertices are labeled using the notation in \cref{cor:gpd} with $\ell=3$.} \label{fig:gpd}
	\end{figure}

	\begin{corollary} \label{cor:gpd}
		Let $\pd$ be a persistence diagram with corresponding $k$-th graded persistence diagram $\pd_k$. Then there exist $\ell \geq 0$, $m_1,\ldots,m_{\ell} \geq 1$,
		$a_{1,1}, \ldots, a_{1,m_1}, a_{2,1}, \ldots, a_{2,m_2}, \ldots, a_{\ell,1}, \ldots, a_{\ell,m_{\ell}} \in \mathbb{R}$,
		$b_{1,1}, \ldots, b_{1,m_1}, b_{2,1}, \ldots, b_{2,m_2}, \ldots, b_{\ell,1},\ldots,b_{\ell,m_{\ell}} \in \mathbb{R}$, (depending on $k$)
		such that
		\begin{equation} \label{eq:gpd}
			\begin{split}
				\pd_k &= [a_{1,1},b_{1,1}) - [a_{1,2},b_{1,1}) + [a_{1,2},b_{1,2}) - \cdots + [a_{1,m_1},b_{1,m_1}) \\
				& + [a_{2,1},b_{2,1}) - [a_{2,2},b_{2,1}) + [a_{2,2},b_{2,2}) - \cdots + [a_{2,m_2},b_{2,m_2}) + \cdots \\
				& + [a_{\ell,1},b_{\ell,1}) - [a_{\ell,2},b_{\ell,1}) + [a_{\ell,2},b_{\ell,2}) - \cdots + [a_{\ell,m_\ell},b_{\ell,m_\ell}),
			\end{split}
		\end{equation}
		the inequalities $a_{1,1} < a_{1,2}< \cdots < a_{1,m_1} < a_{2,1} < a_{2,2} < \cdots < a_{\ell,m_\ell}$ and
		$b_{1,1} < b_{1,2}< \cdots < b_{1,m_1} < b_{2,1} < b_{2,2} < \cdots < b_{\ell,m_\ell}$ hold,
		and for $1 \leq i \leq \ell,$ the inequalities
		$a_{i,1} < b_{i,1}$
		and
		$a_{i,j+1} < b_{i,j}$ hold for $1 \leq j \leq m_i-1$.
		Also for $1 \leq i \leq \ell -1$, $b_{i,m_i} \leq a_{i+1,1}$ holds.
		Furthermore, each $a_{i,j}$ is the first coordinate of an element in $\pd$ and
		each $b_{i,j}$ is the second coordinate of an element in $\pd$.
	\end{corollary}
	
	\begin{remark}\label{rem:cont-gpd}
		Suppose $\overline{M}:\R\to\vect_K$ corresponds to $M:[m]\to\vect_K$ via the order-preserving and injective map $\iota:[m+1]\to\R$ as in \cref{sec:discrete-cont}. By \cref{def:continuous_diagrams} it follows that \cref{prop:pdk} and \cref{cor:gpd} hold for $\pd_k(\overline{M})$.
	\end{remark}
	
	\section{Persistence landscape and derivative persistence landscape}
	\label{sec:pl-dpl}
	
	Using the graded persistence diagram we easily obtain the persistence landscape, its derivative, and its basic properties.
	
	\subsection{The persistence landscape}
	\label{sec:pl}

	Let $\mr_+$ be the subset of $\mr$ given by $\mr_+ = \{x \in \mr \ | \ x >0\}$, and let $\R_+ = (\mr_+, \leq)$ be the corresponding sub-poset of $\R = (\mr,\leq)$.
	For $t \in \mr$, let $\iota_t: \mr_+ \to \mr^2_<$ be given by $h \mapsto [t-h,t+h)$.
	This gives an inclusion of posets $\iota_t:\R_+ \hookrightarrow \R^2_<$.
	It follows that the composition $\Rank_k \circ \iota_t$ is an order-reversing map from $\R_+$ to $(\{0,1\},\leq)$. 
	
	\begin{definition} \label{def:pl}
		Given a persistence module $M$  indexed by $[m]$ and an order-preserving and injective map $\iota:[m+1]\to\R$ let $\overline{M}$ be persistence module indexed by $\R$ corresponding to $M$ via $\iota$ as in \cref{sec:discrete-cont}. We define the \emph{persistence landscape} of $\overline{M}$ to be given by the following. For $k \geq 1$ and $t \in \mr$,
		let $\lambda_k(t) = \sup \{ h>0 \ | \ \Rank_k(\overline{M}) \iota_t (h) = 1 \}$, where $\lambda_k(t) = 0$ if this set is empty.
	\end{definition}
	
	Observe that
	\begin{align*}
		\lambda_k(t) &= \sup( h>0 \mid \Rank_k(\overline{M}) \iota_t(h) = 1)\\
		&= \sup ( h>0 \mid \Rank(\overline{M})[t-h,t+h) \geq k) \\
		&= \sup ( h>0 \mid \rank(\overline{M})[t-h,t+h] \geq k)\text{.}
	\end{align*}
	So \cref{def:pl} agrees with the definition in \cref{sec:pl-background}.

	\subsection{Properties of the persistence landscape}
	\label{sec:properties}
	
	By \cref{def:pl}, as $t$ varies, $[t-\lambda_k(t),t+\lambda_k(t))$ traces out the boundary of the support of $\Rank_k$.
	
	\begin{theorem} \label{thm:properties-pl}
		\begin{enumerate}
			\item $\lambda_k$ is a continuous piecewise-linear function.
			\item The value $\lambda_k(t)$, denoted $h$, is a local maximum at $t$ if and only if $\pd_k[t-h,t+h) = 1$. $\lambda_k(t)$ is a local minimum at $t$ if and only if $\pd_k[t-h,t+h) = -1$.
			\item If $\lambda_k'(t)$ exists then $\lambda_k'(t) \in \{-1,0,1\}$ and $\lambda_k'(t) = 0$ implies that $\lambda_k(t)=0$.
		\end{enumerate}
	\end{theorem}
	
	\begin{proof}
		Since $\Rank_k(\overline{M}) = \zeta * \pd_k(\overline{M})$, it follows that the support of $\Rank_k(\overline{M})$ equals the downward closure of the support of $\pd_k(\overline{M})$. Together with \cref{prop:pdk} and \cref{rem:cont-gpd}, we obtain the desired result.
	\end{proof}
	
	\subsection{The derivative of the persistence landscape}
	\label{sec:dpl}
	
	Write $\pd_k = \sum_{i=1}^n c_i [a_i,b_i)$, where $[a_i,b_i) \in \R^2_<$ and $c_i \in \{-1,1\}$. Let $m_i = \frac{a_i+b_i}{2}$. For $[a,b) \in \Dgm$, let $\chi_{(a,b)}$ denote the indicator function with domain $\mr$ of the subset $(a,b)$.
	
	\begin{definition} \label{def:rhok}
		Define the function $\rho_k: \mr \to \mr$ by 
		\begin{equation} \label{eq:rhok}
			\rho_k(t) = \sum_{i=1}^n c_i \left( \chi_{(a_i,m_i)} - \chi_{(m_i,b_i)} \right).
		\end{equation}
	\end{definition}
	
	First we simplify \eqref{eq:rhok} in a basic example. See \cref{fig:rhok1}. 
	
	\begin{lemma} \label{lem:rhok1}
		If $\pd_k=[a_1,b_1) - [a_2,b_2) + [a_3,b_3)$ and both $a_1 < a_2=a_3$ and $b_1=b_2 < b_3$ hold, then $\rho_k(t) = \chi_{(a_1,m_1)} - \chi_{(m_1,m_2)} + \chi_{(m_2,m_3)} - \chi_{(m_3,b_3)}$.
	\end{lemma}
	
	\begin{proof}
		From \cref{def:rhok} we have,
		\begin{align*}
			\rho_k(t)
			&= \chi_{(a_1,m_1)} - \chi_{(m_1,b_1)}
			- \chi_{(a_2,m_2)} + \chi_{(m_2,b_2)}
			+ \chi_{(a_3,m_3)} - \chi_{(m_3,b_3)}\\
			&= \chi_{(a_1,m_1)} - \chi_{(m_1,b_1)}
			- \chi_{(a_3,m_2)} + \chi_{(m_2,b_1)}
			+ \chi_{(a_3,m_3)} - \chi_{(m_3,b_3)}\\
			&= \chi_{(a_1,m_1)} - \chi_{(m_1,m_2]}
			+ \chi_{[m_2,m_3)} - \chi_{(m_3,b_3)}\\
			&= \chi_{(a_1,m_1)} - \chi_{(m_1,m_2)}
			+ \chi_{(m_2,m_3)} - \chi_{(m_3,b_3)}. \qedhere
		\end{align*}
	\end{proof}

	\begin{figure}
		\begin{center}
			\begin{tikzpicture}[line cap=round,line join=round,>=triangle 45,x=1.0cm,y=1.0cm,scale=0.5]
			\draw[color=black] (0,0) -- (14,0);
			\draw[color=black] (1,0) -- (4,3) -- (7,0);
			\draw[color=black] (3,0) -- (8,5) -- (13,0);
			\foreach \x in {1,3,4,5,7,8,13}
			\draw[shift={(\x,0)}] (0pt,2pt) -- (0pt,-2pt);
			\draw[shift={(1,0)}] node[below] {\footnotesize $a_1$};
			\draw[shift={(3,0)}] node[below] {\footnotesize $a_2$};
			\draw[shift={(3,0)}] node[below,inner sep=1.5em] {\footnotesize $a_3$};
			\draw[shift={(4,0)}] node[below] {\footnotesize $m_1$};
			\draw[shift={(5,0)}] node[below] {\footnotesize $m_2$};
			\draw[shift={(7,0)}] node[below] {\footnotesize $b_1$};
			\draw[shift={(7,0)}] node[below,inner sep=1.5em] {\footnotesize $b_2$};
			\draw[shift={(8,0)}] node[below] {\footnotesize $m_3$};
			\draw[shift={(13,0)}] node[below] {\footnotesize $b_3$};
			\fill [color=blue] (4,3) circle (2mm);
			\draw [very thick,green=blue] (5,2) circle (2mm);
			\fill [color=blue] (8,5) circle (2mm);
			\end{tikzpicture}
			\quad
			\begin{tikzpicture}[line cap=round,line join=round,>=triangle 45,x=1.0cm,y=1.0cm,scale=0.5]
			\draw[color=black] (0,0) -- (14,0);
			\foreach \x in {1,3,4,5,7,8,13}
			\draw[shift={(\x,0)}] (0pt,2pt) -- (0pt,-2pt);
			\draw[shift={(1,0)}] node[below] {\footnotesize $a_1$};
			\draw[shift={(3,0)}] node[below] {\footnotesize $a_2$};
			\draw[shift={(3,0)}] node[below,inner sep=1.5em] {\footnotesize $a_3$};
			\draw[shift={(4,0)}] node[below] {\footnotesize $m_1$};
			\draw[shift={(5,0)}] node[below] {\footnotesize $m_2$};
			\draw[shift={(7,0)}] node[below] {\footnotesize $b_1$};
			\draw[shift={(7,0)}] node[below,inner sep=1.5em] {\footnotesize $b_2$};
			\draw[shift={(8,0)}] node[below] {\footnotesize $m_3$};
			\draw[shift={(13,0)}] node[below] {\footnotesize $b_3$};
			\draw[dashed,thick] (1,2.2) -- (4,2.2);
			\draw[dashed,thick] (4,-2.2) -- (7,-2.2);
			\draw[dotted,thick] (3,-2.4) -- (5,-2.4);
			\draw[dotted,thick] (5,2.2) -- (7,2.2);
			\draw[dashdotted,thick] (3,2.4) -- (8,2.4);
			\draw[dashdotted,thick] (8,-2.2) -- (13,-2.2);
			\draw[thick] (1,2) -- (4,2);
			\draw[thick] (4,-2) -- (5,-2);
			\draw[thick] (5,2) -- (8,2);
			\draw[thick] (8,-2) -- (13,-2);
			\end{tikzpicture}
		\end{center}
		\caption{An example for \cref{def:rhok} and \cref{lem:rhok1}. On the left we have a $k$-th graded persistence diagram, $\pd_k$, which has been rotated clockwise by $\frac{\pi}{4}$. Dark blue disks indicate where $\pd_k$ evaluates to 1 and the light green circle indicates where $\pd_k$ evaluates to -1. On the right we have the graph of terms in the right hand side of \eqref{eq:rhok}, where the first, second, and third terms are dashed, dotted, and dash-dotted, respectively. The graph of their sum, $\rho_k$, is solid. Note that the graphs have been shifted slightly in the vertical direction for ease of visualization.} 
		\label{fig:rhok1}
	\end{figure}
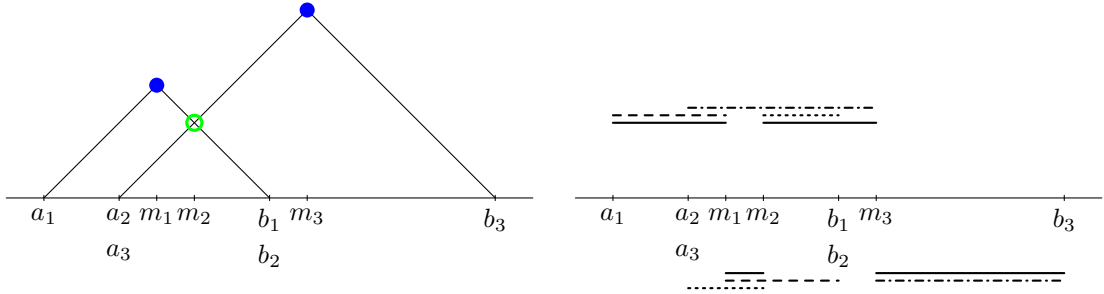

	Next we simplify \eqref{eq:rhok} in a more general example. See \cref{fig:rhok2}.
	
	\begin{lemma} \label{lem:rhok2}
		If $\pd_k = [a_1,b_1) - [a_2,b_2) + [a_3,b_3) - \cdots - [a_{2n},b_{2n}) + [a_{2n+1},b_{2n+1})$ and for $1 \leq i \leq n$ both $a_{2i-1} < a_{2i}=a_{2i+1}$ and $b_{2i-1}=b_{2i} < b_{2i+1}$ hold, then
		\begin{equation*}
			\rho_k(t) = \chi_{(a_1,m_1)} - \chi_{(m_1,m_2)} + \chi_{(m_2,m_3)} - \chi_{(m_3,m_4)} + \cdots + \chi_{(m_{2n},m_{2n+1})} - \chi_{(m_{2n+1},b_{2n+1})}.
		\end{equation*}
	\end{lemma}
	
	\begin{proof}
		From \cref{def:rhok} we have
		\begin{align*}
			\rho_k(t)
			&= \chi_{(a_1,m_1)} - \chi_{(m_1,b_1)}
			- \chi_{(a_2,m_2)} + \chi_{(m_2,b_2)} + \cdots
			+ \chi_{(a_{2n+1},m_{2n+1})} - \chi_{(m_{2n+1},b_{2n+1})}\\
			&= \chi_{(a_1,m_1)} - \chi_{(m_1,b_1)}
			- \chi_{(a_3,m_2)} + \chi_{(m_2,b_1)} + \cdots
			+ \chi_{(a_{2n+1},m_{2n+1})} - \chi_{(m_{2n+1},b_{2n+1})}\\
			&= \chi_{(a_1,m_1)} - \chi_{(m_1,m_2]} + \cdots
			+ \chi_{[m_{2n},m_{2n+1})} - \chi_{(m_{2n+1},b_{2n+1})} \\
			&= \chi_{(a_1,m_1)} - \chi_{(m_1,m_2)} + \cdots
			+ \chi_{(m_{2n},m_{2n+1})} - \chi_{(m_{2n+1},b_{2n+1})} \qedhere
		\end{align*}
	\end{proof}
	
	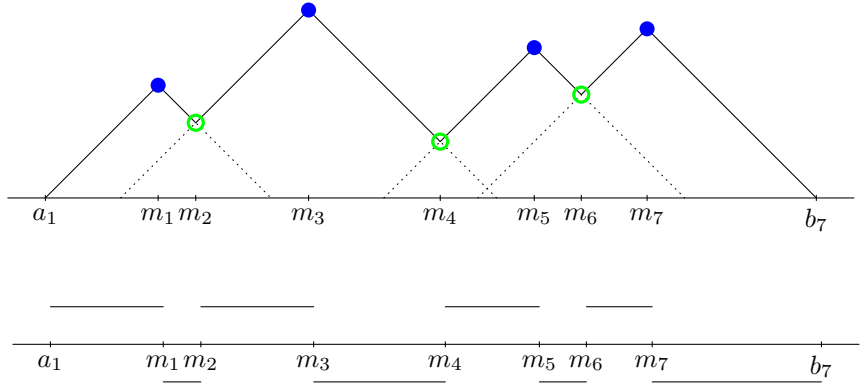
\begin{figure}
		\begin{center}
			\begin{tikzpicture}[line cap=round,line join=round,>=triangle 45,x=1.0cm,y=1.0cm,scale=0.5]
			\draw[color=black] (0,0) -- (22.5,0);
			\draw (1,0) -- (4,3) -- (5,2) -- (8,5) -- (11.5,1.5) -- (14,4) -- (15.25,2.75) -- (17,4.5) -- (21.5,0);
			\draw[dotted] (3,0) -- (5,2) -- (7,0);
			\draw[dotted] (10,0) -- (11.5,1.5) -- (13,0);
			\draw[dotted] (12.5,0) -- (15.25,2.75) -- (18,0);
			\foreach \x in {1,4,5,8,11.5,14,15.25,17,21.5}
			\draw[shift={(\x,0)}] (0pt,2pt) -- (0pt,-2pt);
			\draw[shift={(1,0)}] node[below] {\footnotesize $a_1$};
			\draw[shift={(4,0)}] node[below] {\footnotesize $m_1$};
			\draw[shift={(5,0)}] node[below] {\footnotesize $m_2$};
			\draw[shift={(8,0)}] node[below] {\footnotesize $m_3$};
			\draw[shift={(11.5,0)}] node[below] {\footnotesize $m_4$};
			\draw[shift={(14,0)}] node[below] {\footnotesize $m_5$};
			\draw[shift={(15.25,0)}] node[below] {\footnotesize $m_6$};
			\draw[shift={(17,0)}] node[below] {\footnotesize $m_7$};
			\draw[shift={(21.5,0)}] node[below] {\footnotesize $b_7$};
			\fill [color=blue] (4,3) circle (2mm);
			\draw [very thick,green=blue] (5,2) circle (2mm);
			\fill [color=blue] (8,5) circle (2mm);
			\draw [very thick,green=blue] (11.5,1.5) circle (2mm);
			\fill [color=blue] (14,4) circle (2mm);
			\draw [very thick,green=blue] (15.25,2.75) circle (2mm);
			\fill [color=blue] (17,4.5) circle (2mm);
			\end{tikzpicture}
			\vspace{2em}
			
			\begin{tikzpicture}[line cap=round,line join=round,>=triangle 45,x=1.0cm,y=1.0cm,scale=0.5]
			\draw[color=black] (0,0) -- (22.5,0);
			\foreach \x in {1,4,5,8,11.5,14,15.25,17,21.5}
			\draw[shift={(\x,0)}] (0pt,2pt) -- (0pt,-2pt);
			\draw[shift={(1,0)}] node[below] {\footnotesize $a_1$};
			\draw[shift={(4,0)}] node[below] {\footnotesize $m_1$};
			\draw[shift={(5,0)}] node[below] {\footnotesize $m_2$};
			\draw[shift={(8,0)}] node[below] {\footnotesize $m_3$};
			\draw[shift={(11.5,0)}] node[below] {\footnotesize $m_4$};
			\draw[shift={(14,0)}] node[below] {\footnotesize $m_5$};
			\draw[shift={(15.25,0)}] node[below] {\footnotesize $m_6$};
			\draw[shift={(17,0)}] node[below] {\footnotesize $m_7$};
			\draw[shift={(21.5,0)}] node[below] {\footnotesize $b_7$};
			\draw (1,1) -- (4,1);
			\draw (4,-1) -- (5,-1);
			\draw (5,1) -- (8,1);
			\draw (8,-1) -- (11.5,-1);
			\draw (11.5,1) -- (14,1);
			\draw (14,-1) -- (15.25,-1);
			\draw (15.25,1) -- (17,1);
			\draw (17,-1) -- (21.5,-1);
			\end{tikzpicture}
		\end{center}
		\caption{An example for \cref{def:rhok} and \cref{lem:rhok2}. Above, we have a rotated $k$-th graded persistence diagram. Below, we have the graph of the corresponding function, $\rho_k$. Above, we also have the graph of the $k$-th persistence landscape, $\lambda_k$ (solid). We see that the derivative of $\lambda_k$ is $\rho_k$ (\cref{prop:rhok-lambdak}) and that $\lambda_k$ may be obtained by integrating $\rho_k$ (\cref{cor:dpl}).}
		\label{fig:rhok2}
	\end{figure}
	
	Finally we simplify \eqref{eq:rhok} in the general case.
	
	\begin{theorem} \label{prop:rhok-lambdak}
		The function $\rho_k$ is the derivative of the $k$-th persistence landscape. That is,
		\[
		\rho_k(t) =
		\begin{cases}
		\lambda_k'(t) & \text{if } \lambda_k'(t) \text{ is defined} \\
		0 & \text{ otherwise.}
		\end{cases}
		\]
	\end{theorem}
	
	\begin{proof}
		Since the support of $\Rank_k$ is downwards closed, each midpoint $m_i$ is distinct.
		Order the points in the support of $\pd_k$ so that $m_1 < m_2 < \cdots < m_n$ holds.
		From \cref{prop:pdk} and \cref{rem:cont-gpd}, it follows that if $c_i=-1$ then $c_{i-1} = 1 =  c_{i+1}$, $b_{i-1} = b_i$ and $a_i = a_{i+1}$.
		
		Thus, we have that $(c_1,\ldots,c_n) = (c_1,\ldots,c_{j_1},c_{j_1+1},\ldots,c_{j_2},\ldots,c_{j_m})$ and for $0\leq k \leq m-1$ with $j_0=0$,
		$(c_{j_k+1},\ldots,c_{j_{k+1}}) = (1,-1,1,-1,1,\ldots,1)$.
		Therefore, by \cref{def:rhok} and \cref{lem:rhok2} we have the following.
		\begin{align*}
			\rho_k(t)
			&= \sum_{k=0}^{m-1} \sum_{i=j_k+1}^{j_{k+1}} c_i (\chi_{(a_i,m_i)} - \chi_{(m_i,b_i)})\\
			&= \sum_{k=0}^{m-1} \chi_{(a_{j_k+1},m_{j_k+1})} - \chi_{(m_{j_k+1},m_{j_k+2})} +
			\cdots + \chi_{(m_{j_{k+1}-1},m_{j_{k+1}})} - \chi_{(m_{j_{k+1}},b_{j_{k+1}})}\\
		\end{align*}
		This sum of indicator functions is precisely $\lambda'_k$ where $\lambda'_k$ is defined and is otherwise $0$.
	\end{proof}
	
	\begin{corollary} \label{cor:dpl}
		$\lambda_k(t) = \int_{-\infty}^t \rho_k(s) ds$
	\end{corollary}
	
	\begin{proof}
		Since $\lambda_k$ and $\rho_k$ have bounded support, the result follows from \cref{prop:rhok-lambdak}.
	\end{proof}

	\section{Wasserstein stability for graded persistence diagrams}
	\label{sec:wasserstein}
	
	In this section we define a Wasserstein distance for graded persistence diagrams and prove that the mapping from an ungraded persistence diagram to a graded persistence diagram is stable. Throughout we will continue to assume persistence diagrams and graded persistence diagrams are tame as detailed in \cref{sec:discrete-cont,sec:cont}.
	
	\subsection{Wasserstein distance for graded persistence diagrams}
	
	We start by recalling the Wasserstein distance for persistence diagrams. Let $\Delta = \{(x,x)\in\mr^2\}$ and let $p,q \in [1,\infty]$.
	
	\begin{definition} \label{def:coupling}
		Let $D,E:\Dgm\to\mz_{\geq 0}$ be persistence diagrams. A \emph{coupling} between $D$ and $E$ is a map $\gamma:(\Dgm\cup\Delta) \times (\Dgm\cup\Delta) \to \mz_{\geq 0}$ where
		$\gamma(z,w) = 0$ for all $(z,w) \in \Delta \times \Delta$ and
		for all $z\in\Dgm$
		\[ 
		D(z) = \sum_{w\in\Dgm \cup \Delta} \gamma(z,w)
		\]
		and for all $w\in\Dgm$
		\[
		E(w) = \sum_{z\in\Dgm\cup\Delta}\gamma(z,w)\text{.}
		\]
		Note that $\gamma$ is a multiset on $(\Dgm\cup\Delta) \times (\Dgm\cup\Delta)$. Since $\abs{D}$ and $\abs{E}$ are finite, so is $\abs{\gamma}$.
		The \emph{$(p,q)$-cost} of $\gamma$ is
		$\norm{\gamma}_{p,q} = \norm{\left(\norm{w-z}_q \ | \ (z,w) \in \gamma\right)}_p$. The notation treats $\gamma$ as a multiset and the elements of $\gamma$ are taken with multiplicity.
		That is, we take the $p$-norm of the vector whose entries consist of the distances in the $q$-norm between $z$ and $w$ for all pairs $(z,w)$ in the multiset $\gamma$.
	\end{definition}
	
	\begin{definition} \label{def:wasserstein}
		The \emph{$(p,q)$-Wasserstein distance}, $W_{p,q}$ between persistence diagrams $D$ and $E$ is given by
		\[ 
		W_{p,q}(D,E)=\inf\norm{\gamma}_{p,q}
		\]
		where the infimum is taken over all couplings of $D$ and $E$.
	\end{definition}
	
	\begin{remark} \label{rem:wasserstein}
		By our finiteness assumption on persistence modules our persistence diagrams have finite support. 
		Since we aim to minimize cost, we may assume that if $\gamma((x,y),(z,w)) \neq 0$ where $(x,y) \in \Dgm$ and $(z,w) \in \Delta$ then $(z,w) = (\frac{x+y}{2},\frac{x+y}{2})$.
		Similarly if $(x,y) \in \Delta$ and $(z,w) \in \Dgm$ then we may assume that $(x,y) = (\frac{z+w}{2},\frac{z+w}{2})$.
		Under this assumption and our finiteness assumption, there are only finitely many possible couplings.
		Since our persistence diagrams are finite multisets, they may be equivalently represented as finite indexed sets. That is $D = \{(x_i,y_i)\}_{i=1}^{m}$ and $E = \{(z_i,w_i)\}_{i=1}^{n}$, where for all $i$, $(x_i,y_i),(z_i,w_i) \in \Dgm$.
		For each coupling $\gamma$, let $r$ be the cardinality of $\gamma$ restricted to $\Dgm \times \Dgm$. 
		Then the cardinality of $\gamma$ restricted to $\Dgm \times \Delta$ is $m-r$, which we denote by $s$, and
		the cardinality of $\gamma$ restricted to $\Delta \times \Dgm$ is $n-r$, which we denote by $t$.
		Therefore, for each $\gamma$ we may choose an ordering of the indexed sets $D$ and $E$ such that 
		\begin{multline*}
			\norm{\gamma}_{p,q} = 
			\left\lVert \left( \norm{(x_1,y_1)-(z_1,w_1)}_q,\ldots,\norm{(x_{r},y_{r})-(z_{r},w_{r})}_q,\right. \right. \\
			\norm{(x_{r+1},y_{r+1})-(\tfrac{x_{r+1}+y_{r+1}}{2},\tfrac{x_{r+1}+y_{r+1}}{2})}_q,\ldots,\norm{(x_{r+s},y_{r+s})-(\tfrac{x_{r+s}+y_{r+s}}{2},\tfrac{x_{r+s}+y_{r+s}}{2})}_q,\\
			\left. \left.
			\norm{(z_{r+1},w_{r+1})-(\tfrac{z_{r+1}+w_{r+1}}{2},\tfrac{z_{r+1}+w_{r+1}}{2})}_q,\ldots,\norm{(z_{r+t},w_{r+t})-(\tfrac{z_{r+t}+w_{r+t}}{2},\tfrac{z_{r+t}+w_{r+t}}{2})}_q \right) \right\rVert_p.
		\end{multline*}
	\end{remark}
	One may check that this definition agrees with the typical definition (e.g. \cite[pp.216,219--220]{HarerBook}) upon viewing persistence diagrams as multisets and couplings as matchings. The following is straightforward to check from the definition.
	
	\begin{proposition}
		The $(p,q)$-Wasserstein distance is a metric for persistence diagrams.
	\end{proposition}
	
	Recall that the $k$-th graded persistence diagram is a function
	$D_k:\Dgm \to \mz$ with finite support.
	Consider a function $A:\Dgm \to \mz$ with finite support.
	Then there exist unique persistence diagrams  $A^+,A^-:\Dgm \to \mz_{\geq 0}$ with disjoint support such that
	$A = A^+ - A^-$. The following definition is due to Bubenik and Elchesen~\cite{be:virtual}
	and applies to $k$-th graded persistence diagrams $D_k$ and $E_k$.
	
	\begin{definition} \label{def:graded-wasserstein}
		Let $A,B:\Dgm \to \mz$ be functions with finite support.
		Define the \emph{$(p,q)$-Wasserstein distance} between $A$ and $B$ to be given by 
		\begin{equation*}
			W_{p,q}(A,B) := W_{p,q}(A^+ + B^-, B^+ + A^-).
		\end{equation*}
	\end{definition}
	
	In their manuscript~\cite{be:virtual}, it is shown that if $W_{p,q}$ satisfies the condition that $W_{p,q}(D+F,E+F) = W_{p,q}(D,E)$ for all $D,E,F$ then $W_{p,q}$ is a metric and furthermore that this condition holds if $p=1$.
	
	\begin{proposition}[\cite{be:virtual}] \label{prop:metric}
		The $(1,q)$-Wasserstein distance is a metric for
		functions from $\Dgm$ to $\mz$ with finite support (e.g. persistence diagrams, and $k$-th graded persistence diagrams).
		Furthermore, for all such functions $D,E,F$, $W_{1,q}(D+F,E+F) = W_{1,q}(D,E)$.
	\end{proposition}
	
	Here we show that in all other cases the triangle inequality is not satisfied. Recall that the interval $[x,y)\in\Dgm$ denotes the persistence diagram that takes value $1$ on $[x,y)$ and $0$ elsewhere.
	
	\begin{proposition} \label{prop:triangle}
		For $1 < p \leq \infty$ and $1 \leq q \leq \infty$ the $(p,q)$-Wasserstein distance for $k$-th graded persistence diagrams does not satisfy the triangle inequality.
	\end{proposition}
	
	\begin{proof}
		Consider the persistence diagram $D$ given by $D = [0,10) + (k-1)[0,12)$ and thus $D_k=[0,10)$. Also consider for $0 < \eps \leq 1$ the pair of persistence diagrams $F,E:\Dgm\to\mz$ given by $F=[2,10+2\eps) + (k-1)[0,12)$ and $E=[0,10) + [1,10+\eps) + [2,10+2\eps) + (k-1)[0,12)$. Thus $F_k=[2,10+2\eps)$ and $E_k=[0,10) - [1,10) + [1,10+\eps) - [2,10+\eps) + [2,10+2\eps)$, respectively.
		
		Notice that $W_{p,q}(D_k,F_k) = \norm{(2,2\eps)}_q = 2\norm{(1,\eps)}_q$, which is realized by the coupling which matches $[0,10)$ to $[2,10+2\eps)$. 
		Also observe that 
		$W_{p,q}(D_k,E_k) = W_{p,q}([0,10)+[1,10)+[2,10+\eps),[0,10)+[1,10+\eps)+[2,10+2\eps)) = \norm{(\eps,\eps)}_p = \eps \norm{(1,1)}_p$ via the coupling that matches $[0,10)$ to $[0,10)$, $[1,10)$ to $[1,10+\eps)$, and $[2,10+\eps)$ to $[2,10+2\eps)$. Similarly, we have $W_{p,q}(E_k,F_k) = \norm{(1,1)}_p$.
		Assume, to the contrary, that $W_{p,q}$ satisfies the triangle inequality for $k$-th graded persistence diagrams. Then we have $2 \leq 2\norm{(1,\eps)}_q = W_{p,q}(D_k,F_k) \leq W_{p,q}(D_k,E_k) + W_{p,q}(E_k,F_k) = (1+\eps) \norm{(1,1)}_p $ for all $\eps$ with $0 < \eps \leq 1$.
		Therefore $2 \leq \norm{(1,1)}_p$,
		which contradicts that $p>1$.
	\end{proof}
	\begin{figure}[h]
		\begin{center} 
			\begin{tikzpicture}[line cap=round,line join=round,>=triangle 45,x=1.0cm,y=1.0cm,scale=0.7]
			\draw[color=black] (0,8) -- (6,8);
			\foreach \x in {1,2,...,6}
			\draw[shift={(\x,8)},color=black] (0pt,2pt) -- (0pt,-2pt);
			\foreach \x in {1,2,3,...,6}
			\draw[shift={(\x,8)},color=black] node[below] {\footnotesize $\x$};
			\draw[color=black] (0,8) -- (0,12.5);
			\foreach \y in {8,9,...,12}
			\draw[shift={(0,\y)},color=black] node[left] {\footnotesize $\y$};

			\fill [color=blue] (0,10) circle (2mm);
			\draw [mark=triangle*,mark options={color=blue},mark size=2mm] plot coordinates {(2,11)};
			\draw[color=orange,very thick] (0,10) -- (2,11);
			\end{tikzpicture}
			\begin{tikzpicture}[line cap=round,line join=round,>=triangle 45,x=1.0cm,y=1.0cm,scale=0.7]
			\draw[color=black] (0,8) -- (6,8);
			\foreach \x in {1,2,...,6}
			\draw[shift={(\x,8)},color=black] (0pt,2pt) -- (0pt,-2pt);
			\foreach \x in {1,2,3,...,6}
			\draw[shift={(\x,8)},color=black] node[below] {\footnotesize $\x$};
			\draw[color=black] (0,8) -- (0,12.5);
			\foreach \y in {8,9,...,12}
			\draw[shift={(0,\y)},color=black] node[left] {\footnotesize $\y$};
			
			\fill [color=blue,shift={(0,.2)}] (0,10) circle (2mm);
			
			\draw [mark=square*,mark options={color=blue},mark size=3pt] plot coordinates {(0,10)};
			
			\draw [mark=square,mark options={color=green},mark size=3pt] plot coordinates {(1,10)};
			
			\draw [mark=square*,mark options={color=blue},mark size=3pt] plot coordinates {(1,10.5)};
			
			\draw [mark=square,mark options={color=green},mark size=3pt] plot coordinates {(2,10.5)};
			
			\draw [mark=square*,mark options={color=blue},mark size=3pt] plot coordinates {(2,11)};
			
			\draw[color=orange,very thick] (0,10) to[out=0,in=90,distance=40pt] (0,10.2);
			\draw[color=orange,very thick] (1,10) to[out=90,in=270,distance=20pt] (1,10.5);
			\draw[color=orange,very thick] (2,10.5) to[out=90,in=270,distance=20pt] (2,11);
			\end{tikzpicture}
			\begin{tikzpicture}[line cap=round,line join=round,>=triangle 45,x=1.0cm,y=1.0cm,scale=0.7]
			\draw[color=black] (0,8) -- (6,8);
			\foreach \x in {1,2,...,6}
			\draw[shift={(\x,8)},color=black] (0pt,2pt) -- (0pt,-2pt);
			\foreach \x in {1,2,3,...,6}
			\draw[shift={(\x,8)},color=black] node[below] {\footnotesize $\x$};
			\draw[color=black] (0,8) -- (0,12.5);
			\foreach \y in {8,9,...,12}
			\draw[shift={(0,\y)},color=black] node[left] {\footnotesize $\y$};
			
			\draw [mark=triangle*,mark options={color=blue},mark size=4pt,shift={(-.10,0)}] plot coordinates {(2,11)};
			
			\draw [mark=square*,mark options={color=blue},mark size=3pt] plot coordinates {(0,10)};
			
			\draw [mark=square,mark options={color=green},mark size=3pt] plot coordinates {(1,10)};
			
			\draw [mark=square*,mark options={color=blue},mark size=3pt] plot coordinates {(1,10.5)};
			
			\draw [mark=square,mark options={color=green},mark size=3pt] plot coordinates {(2,10.5)};
			
			\draw [mark=square*,mark options={color=blue},mark size=3pt] plot coordinates {(2,11)};
			
			\draw[color=orange,very thick] (1.9,11) to[out=180,in=45,distance=50pt] (2,11);
			\draw[color=orange,very thick] (0,10) to[out=0,in=180,distance=20pt] (1,10);
			\draw[color=orange,very thick] (1,10.5) to[out=0,in=180,distance=20pt] (2,10.5);
			\end{tikzpicture}
		\end{center}
		\caption{Graded persistence diagrams and couplings as in the proof of \cref{prop:triangle} with $\eps=0.5$. Circles denote $D_k$, squares denote $E_k$, and triangles denote $F_k$. Solid blue points evaluate to +1 and hollow green points evaluate to -1. Orange lines between points denote that the points are matched in the coupling. {\bf Left}: Coupling between $D_k$ and $F_k$. {\bf Middle}: Coupling between $D_k$ and $E_k$. {\bf Right}: Coupling between $F_k$ and $E_k$.}\label{fig:trig_inequality}
	\end{figure}
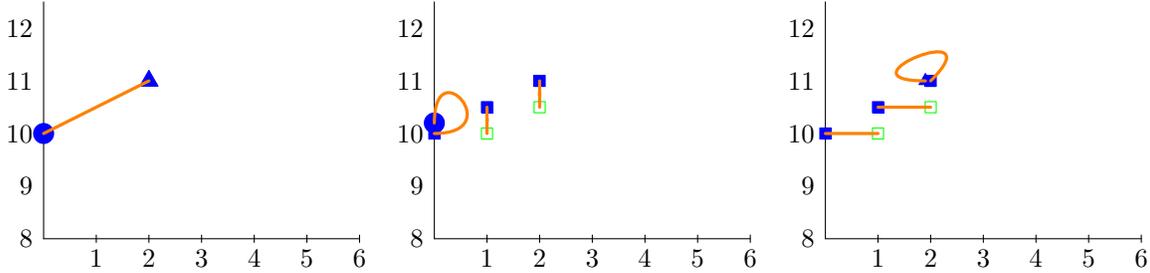
	
	\subsection{Stability of Graded Persistence Diagrams}
	
	We prove a stability theorem for graded persistence diagrams using the Wasserstein distance and certain geodesics. 
	
	Let $M$ and $N$ be persistence modules with
	persistence diagrams $D,E: \Dgm \to \mz_{\geq 0}$ and 
	$k$-th graded persistence diagrams $D_k,E_k: \Dgm \to \mz$ for $k \geq 1$. Recall that for $1 \leq q \leq \infty$ the $q$-norms on $\mr^2$ produce strongly equivalent metrics.
	In our setting, we will use the $1$-norm because it gives a nice expression for the Wasserstein distance~\eqref{eq:wasserstein} and a nice statement for stability (Theorem~\ref{thm:stability}). By \cref{rem:wasserstein}, there is an ordering of the points of $D$ and $E$ such that
	\begin{equation} \label{eq:wasserstein}
		W_{1,1}(D,E) =
		\sum_{i=1}^{r} \abs{x_i - z_i} +
		\sum_{i=1}^{r} \abs{y_i - w_i} +
		\sum_{i=r + 1}^{r+s} (y_i - x_i) +
		\sum_{i=r+1}^{r+t} (w_i - z_i),
	\end{equation}
	where $D = \{(x_i,y_i)\}_{i=1}^{r+s}$ and
	$E = \{(z_i,w_i)\}_{i=1}^{r+t}$.
	
	\begin{definition}
		Let $x$ and $y$ be points in a metric space $(X,d)$ with $\tau := d(x,y)$. A \emph{geodesic} from $x$ to $y$ is a map $\gamma:[0,\tau] \to X$ such that for $0 \leq s \leq t \leq \tau$, $d(\gamma(s),\gamma(t)) = t-s$.
	\end{definition}
	
	We will show that the $(1,1)$-Wasserstein distance between persistence diagrams can be realized by a geodesic.
	Furthermore, considering each persistence diagram as a finite indexed set of points in $\Dgm$ (\cref{rem:wasserstein}), we can choose a geodesic that is a concatenation of finitely many geodesics that leave all but one of the points of the persistence diagram fixed and leave one of the coordinates of the remaining point fixed.
	
	Let $\DGM$ denote the set of persistence diagrams with the $(1,1)$-Wasserstein distance. The following are consequence of \cref{def:wasserstein,def:coupling} and \cref{rem:wasserstein}. Call the geodesics in the lemmas below and their reverses \emph{coordinate geodesics}.
	
	\begin{lemma} \label{lem:geodesic1}
		Let {$D = D' + [x,y)$} be a persistence diagram.
		Choose $z$ with $z < y$.
		Let $\tau = \abs{z-x}$.
		Let {$E = D' +[z,y)$}.
		Let $\gamma:[0,\tau] \to \DGM$ be given by
		{$\gamma(t) = D' + [x_t,y)$}, where $x_t = x(1-\frac{t}{\tau}) + \frac{t}{\tau}z$.
		Then $\gamma$ is a geodesic from $D$ to $E$.
	\end{lemma}
	
	\begin{lemma} \label{lem:geodesic2}
		Let {$D = D' + [x,y)$} be a persistence diagram.
		Choose $w$ with $x<w$.
		Let $\tau = \abs{w-y}$.
		Let $E = D' + [x,w)$.
		Let $\gamma:[0,\tau] \to \DGM$ be given by
		{$\gamma(t) = D' + [x,y_t)$}, where $y_t = y(1-\frac{t}{\tau}) + \frac{t}{\tau}w$.
		Then $\gamma$ is a geodesic from $D$ to $E$.
	\end{lemma}
	
	\begin{lemma} \label{lem:geodesic3}
		Let {$D = D' + [x,y)$} be a persistence diagram.
		Let $\tau = y-x$.
		Let $\gamma:[0,\tau] \to \DGM$ be defined as follows.
		For $0 \leq t < \tau$,
		{$\gamma(t) = D' + [x_t,y)$}, where $x_t = x(1-\frac{t}{\tau}) + \frac{t}{\tau}y$ 
		and $\gamma(\tau) = D'$.
		Then $\gamma$ is a geodesic from $D$ to $D'$.
	\end{lemma}
	\begin{proof}
		{We prove \cref{lem:geodesic1}, the others are similar. For any $s,t$ with $0\leq s \leq t \leq \tau$, $\gamma(s)$ and $\gamma(t)$ are the persistence diagrams $[x_s,y)+D'$ and $[x_t,y)+D'$ respectively. Apply \cref{prop:metric} to obtain $W_{1,1}(\gamma(s),\gamma(t)) = W_{1,1}([x_s,y),[x_t,y)) = t-s$.}
	\end{proof}
	
	\begin{proposition} \label{prop:geodesic}
		Let $D$ and $E$ be persistence diagrams. Then there is a geodesic from $D$ to $E$ consisting of a concatenation of finitely many coordinate geodesics. 
	\end{proposition}
	
	\begin{proof}
		Consider \eqref{eq:wasserstein}. We obtain the desired geodesic by concatenating
		a coordinate geodesic from \cref{lem:geodesic1} for each term in the first sum in \eqref{eq:wasserstein},
		a coordinate geodesic from \cref{lem:geodesic2} for each term in the second sum in \eqref{eq:wasserstein},
		a coordinate geodesic from \cref{lem:geodesic3} for each term in the third sum in \eqref{eq:wasserstein}, and the reverse of
		a coordinate geodesic from \cref{lem:geodesic3} for each term in the fourth sum in \eqref{eq:wasserstein}.
	\end{proof}
	
	Let $M$ and $N$ be persistence modules with
	persistence diagrams $D,E: \Dgm \to \mz_{\geq 0}$ and 
	$k$-th graded persistence diagrams $D_k,E_k: \Dgm \to \mz$ for $k \geq 1$.
	Let $K$ be the maximum of $\rank(M)$ and $\rank(N)$.
	
	\begin{theorem} \label{thm:stability}
		For $1 \leq k < K$, $W_{1,1}(D_k,E_k) \leq 2\, W_{1,1}(D,E)$.
		Also $W_{1,1}(D_K,E_K) \leq W_{1,1}(D,E)$ and for $k > K$, $D_k = E_k = 0$.
		Furthermore, there exist $M$ and $N$ such that all of these bounds are attained.
	\end{theorem}
	
	\begin{proof} 
		Let $D$ and $E$ be persistence diagrams with corresponding $k$-th graded persistence diagrams $D_k$ and $E_k$.
		By \cref{prop:geodesic} and the triangle inequality, we can reduce to the case that there is a coordinate geodesic {$\gamma:[0,\tau]\to\DGM$} from $D$ to $E$.
		Assume the coordinate that varies is the first coordinate.
		The other case is similar.

		{For every $t\in [0,\tau]$} let $\gamma_k(t)$ be the $k$-th persistence diagram of $\gamma(t)$.
		By \cref{cor:gpd}, for each $t$, $\gamma_k(t)$ can be written as in \eqref{eq:gpd}. 
		{Note that it suffices to consider the case where $\gamma:[0,\tau]\to\DGM$ has the following properties for all $t$ and $t'$ with $0 \leq t,t' < \tau:$
			\begin{enumerate}
				\item $\gamma_k(t)$ and $\gamma_k(t')$ have the same form given by \eqref{eq:gpd}.
				\item $\gamma_k(t)$ only differs from $\gamma_k(t')$ in that some particular $a_{i,j}$ is the coordinate that varies.
				\item The coordinate $a_{i,j}$ that varies is constrained by the inequalities below \eqref{eq:gpd}.
			\end{enumerate}
			This follows by observing any coordinate geodesic is a concatenation of geodesics or the reverse of geodesics with the above properties. For any $\gamma$ fulfilling the properties, either $\gamma_k(\tau)$ also has the same form \eqref{eq:gpd} or as $t$ approaches $\tau$, $a_{i,j}$ approaches the limit of a constraint in \cref{cor:gpd}. }
		
		We have the following cases, {where $m_i$ and $m_{i-1}$ are defined as in \eqref{eq:gpd}}  (\cref{fig:stability}{, Left}):
		(1) $a_{i,j} \to a_{i,j+1}$,
		(2) $a_{i,j} \to a_{i,j-1}$,
		(3) $m_i=1$ and $a_{i,1} \to b_{i,1}$,
		(4) $j \geq 2$ and $a_{i,j} \to b_{i,j-1}$, and
		(5) $\gamma_k(\tau)$ has the same form as $\gamma_k(t)$ for $0\leq t < \tau$, which includes the case that
		$i \geq 2$ and $a_{i,j} \to b_{i-1,m_{i-1}}$.

		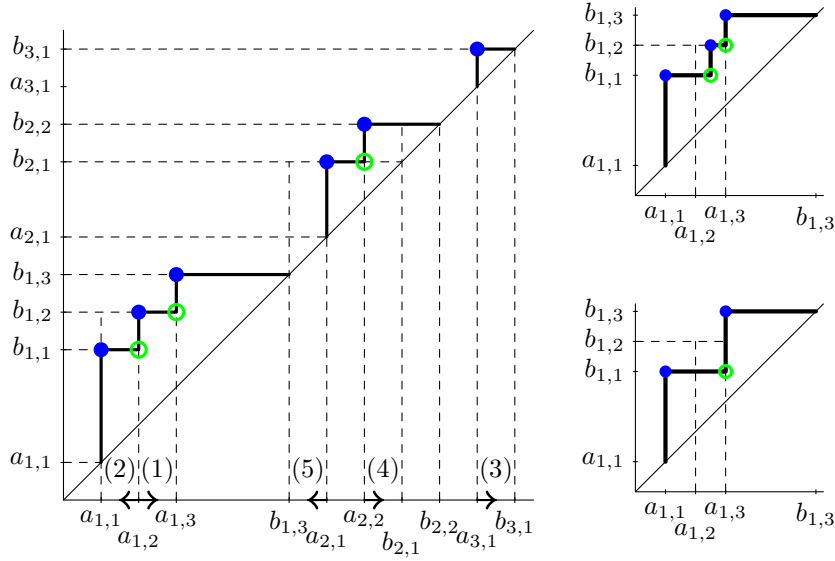
\begin{figure}
			\centering
			\begin{tabular}{cc}
				\multirow[c]{2}{*}[3cm]{%
					\begin{tikzpicture}[line cap=round,line join=round,>=triangle 45,x=1.0cm,y=1.0cm,scale=0.5]
					\draw[color=black] (0,0) -- (12.5,0);
					\draw[color=black] (0,0) -- (0,12.5);
					\draw[color=black] (0,0) -- (12.5,12.5);
					\foreach \x in {1,2,3,6,7,8,9,10,11,12}
					\draw[shift={(\x,0)}] (0pt,2pt) -- (0pt,-2pt);
					\draw[shift={(1,0)}] node[below] {\footnotesize $a_{1,1}$};
					\draw[color=black,dashed,thin] (1,0) -- (1,5);
					\draw[shift={(2,0)}] node[below,inner sep=1em] {\footnotesize $a_{1,2}$};
					\draw[color=black,dashed] (2,0) -- (2,5);
					\draw[shift={(3,0)}] node[below] {\footnotesize $a_{1,3}$};
					\draw[color=black,dashed] (3,0) -- (3,6);	
					\draw[shift={(6,0)}] node[below] {\footnotesize $b_{1,3}$};
					\draw[color=black,dashed] (6,0) -- (6,9);	
					\draw[shift={(7,0)}] node[below,inner sep=1em] {\footnotesize $a_{2,1}$};
					\draw[color=black,dashed] (7,0) -- (7,9);
					\draw[shift={(8,0)}] node[below] {\footnotesize $a_{2,2}$};
					\draw[color=black,dashed] (8,0) -- (8,9);
					\draw[shift={(9,0)}] node[below,inner sep=1em] {\footnotesize $b_{2,1}$};
					\draw[color=black,dashed] (9,0) -- (9,10);	
					\draw[shift={(10,0)}] node[below] {\footnotesize $b_{2,2}$};
					\draw[color=black,dashed] (10,0) -- (10,10);	
					\draw[shift={(11,0)}] node[below,inner sep=1em] {\footnotesize $a_{3,1}$};
					\draw[color=black,dashed] (11,0) -- (11,11);	
					\draw[shift={(12,0)}] node[below] {\footnotesize $b_{3,1}$};
					\draw[color=black,dashed] (12,0) -- (12,12);		  
					\foreach \y in {1,4,5,6,7,9,10,11,12}
					\draw[shift={(0,\y)},color=black] (2pt,0pt) -- (-2pt,0pt);
					\draw[shift={(0,1)}] node[left] {\footnotesize $a_{1,1}$};
					\draw[color=black,dashed] (0,1) -- (1,1);	
					\draw[shift={(0,4)}] node[left] {\footnotesize $b_{1,1}$};
					\draw[color=black,dashed] (0,4) -- (1,4);		
					\draw[shift={(0,5)}] node[left] {\footnotesize $b_{1,2}$};
					\draw[color=black,dashed] (0,5) -- (2,5);		
					\draw[shift={(0,6)}] node[left] {\footnotesize $b_{1,3}$};
					\draw[color=black,dashed] (0,6) -- (3,6);			
					\draw[shift={(0,7)}] node[left] {\footnotesize $a_{2,1}$};
					\draw[color=black,dashed] (0,7) -- (7,7);				
					\draw[shift={(0,9)}] node[left] {\footnotesize $b_{2,1}$};
					\draw[color=black,dashed] (0,9) -- (9,9);					
					\draw[shift={(0,10)}] node[left] {\footnotesize $b_{2,2}$};
					\draw[color=black,dashed] (0,10) -- (9,10);					
					\draw[shift={(0,11)}] node[left] {\footnotesize $a_{3,1}$};
					\draw[shift={(0,12)}] node[left] {\footnotesize $b_{3,1}$};
					\draw[color=black,dashed] (0,12) -- (11,12);					
					\draw[very thick] (1,1) -- (1,4) -- (2,4) -- (2,5) -- (3,5) -- (3,6) -- (6,6);
					\draw[very thick] (7,7) -- (7,9) -- (8,9) -- (8,10) -- (10,10);
					\draw[very thick] (11,11) -- (11,12) -- (12,12);
					\fill [color=blue] (1,4) circle (2mm);
					\draw [very thick,green=blue] (2,4) circle (2mm);
					\fill [color=blue] (2,5) circle (2mm);
					\draw [very thick,green=blue] (3,5) circle (2mm);
					\fill [color=blue] (3,6) circle (2mm);
					\fill [color=blue] (7,9) circle (2mm);
					\draw [very thick,green=blue] (8,9) circle (2mm);
					\fill [color=blue] (8,10) circle (2mm);
					\fill [color=blue] (11,12) circle (2mm);
					\draw[thick,decoration={markings,mark=at position 1 with
						{\arrow[scale=1.5,>=to]{>}}},postaction={decorate}] (2,0) -- (2.5,0);
					\draw[shift={(2.5,0)}] node[above,inner sep=0.5em] {\footnotesize $(1)$};
					\draw[thick,decoration={markings,mark=at position 1 with
						{\arrow[scale=1.5,>=to]{>}}},postaction={decorate}] (2,0) -- (1.5,0);
					\draw[shift={(1.5,0)}] node[above,inner sep=0.5em] {\footnotesize $(2)$};
					\draw[thick,decoration={markings,mark=at position 1 with
						{\arrow[scale=1.5,>=to]{>}}},postaction={decorate}] (11,0) -- (11.5,0);
					\draw[shift={(11.5,0)}] node[above,inner sep=0.5em] {\footnotesize $(3)$};
					\draw[thick,decoration={markings,mark=at position 1 with
						{\arrow[scale=1.5,>=to]{>}}},postaction={decorate}] (8,0) -- (8.5,0);
					\draw[shift={(8.5,0)}] node[above,inner sep=0.5em] {\footnotesize $(4)$};
					\draw[thick,decoration={markings,mark=at position 1 with
						{\arrow[scale=1.5,>=to]{>}}},postaction={decorate}] (7,0) -- (6.5,0);
					\draw[shift={(6.5,0)}] node[above,inner sep=0.5em] {\footnotesize $(5)$};
					\end{tikzpicture}} & \begin{tikzpicture}[line cap=round,line join=round,>=triangle 45,x=1.0cm,y=1.0cm,scale=0.4]
				\draw[color=black] (0,0) -- (6.25,0);
				\draw[color=black] (0,0) -- (0,6.25);
				\draw[color=black] (0,0) -- (6.25,6.25);
				
				\foreach \x in {1,2,3,6}
				\draw[shift={(\x,0)}] (0pt,2pt) -- (0pt,-2pt);
				\draw[shift={(1,0)}] node[below] {\footnotesize $a_{1,1}$};
				\draw[shift={(2,0)}] node[below,inner sep=1em] {\footnotesize $a_{1,2}$};
				\draw[color=black,dashed] (2,0) -- (2,5);				
				\draw[shift={(3,0)}] node[below] {\footnotesize $a_{1,3}$};
				\draw[color=black,dashed] (3,0) -- (3,5);				    
				\draw[shift={(6,0)}] node[below] {\footnotesize $b_{1,3}$};

				\foreach \y in {1,4,5,6}
				\draw[shift={(0,\y)},color=black] (2pt,0pt) -- (-2pt,0pt);
				\draw[shift={(0,1)}] node[left] {\footnotesize $a_{1,1}$};
				\draw[shift={(0,4)}] node[left] {\footnotesize $b_{1,1}$};
				\draw[shift={(0,5)}] node[left] {\footnotesize $b_{1,2}$};
				\draw[color=black,dashed] (0,5) -- (3,5);				    
				\draw[shift={(0,6)}] node[left] {\footnotesize $b_{1,3}$};

				\draw[color=black,ultra thick] (1,1) -- (1,4) -- (2.5,4) -- (2.5,5) -- (3,5) -- (3,6) -- (6,6);

				\fill [color=blue] (1,4) circle (2mm);
				\draw [very thick,green=blue] (2.5,4) circle (2mm);
				\fill [color=blue] (2.5,5) circle (2mm);
				\draw [very thick,green=blue] (3,5) circle (2mm);
				\fill [color=blue] (3,6) circle (2mm);
				\end{tikzpicture} \\
				& \begin{tikzpicture}[line cap=round,line join=round,>=triangle 45,x=1.0cm,y=1.0cm,scale=0.4]
				\draw[color=black] (0,0) -- (6.25,0);
				\draw[color=black] (0,0) -- (0,6.25);
				\draw[color=black] (0,0) -- (6.25,6.25);
				
				\foreach \x in {1,2,3,6}
				\draw[shift={(\x,0)}] (0pt,2pt) -- (0pt,-2pt);
				\draw[shift={(1,0)}] node[below] {\footnotesize $a_{1,1}$};
				\draw[shift={(2,0)}] node[below,inner sep=1em] {\footnotesize $a_{1,2}$};
				\draw[color=black,dashed] (2,0) -- (2,5);				
				\draw[shift={(3,0)}] node[below] {\footnotesize $a_{1,3}$};
				\draw[color=black,dashed] (3,0) -- (3,4);				
				\draw[shift={(6,0)}] node[below] {\footnotesize $b_{1,3}$};

				\foreach \y in {1,4,5,6}
				\draw[shift={(0,\y)},color=black] (2pt,0pt) -- (-2pt,0pt);
				\draw[shift={(0,1)}] node[left] {\footnotesize $a_{1,1}$};
				\draw[shift={(0,4)}] node[left] {\footnotesize $b_{1,1}$};
				\draw[shift={(0,5)}] node[left] {\footnotesize $b_{1,2}$};
				\draw[color=black,dashed] (0,5) -- (3,5);				    
				\draw[shift={(0,6)}] node[left] {\footnotesize $b_{1,3}$};

				\draw[ultra thick] (1,1) -- (1,4) -- (3,4) -- (3,6) -- (6,6);

				\fill [color=blue] (1,4) circle (2mm);
				\draw [very thick,green=blue] (3,4) circle (2mm);
				\fill [color=blue] (3,6) circle (2mm);
				\end{tikzpicture}
			\end{tabular}
			\caption{{\bf Left}: An example of a $k$-th graded persistence diagram $\pd_k$. Dark blue disks indicate where $\pd_k$ evaluates to 1 and light green circles indicate where $\pd_k$ evaluates to -1. Vertices are labeled using the notation in \cref{cor:gpd} with $\ell=3$. Examples of the five cases in the proof of \cref{thm:stability} are indicated with labeled arrows. For example,  in case (1) only the first connected component on the left changes as $t$ changes. {\bf Upper right}: The first connected component of $\gamma_k(t)$ for $0<t<\tau$ in case (1). {\bf Lower right}: The first connected component of $\gamma_k(\tau)$ in case (1). }
			\label{fig:stability}
		\end{figure}

		In case (1),
		let $A_k$ be the points in common in $D_k$ and $E_k$.
		Then $D_k = A_k - [a_{i,j},b_{i,j-1}) + [a_{i,j},b_{i,j}) - [a_{i,j+1},b_{i,j})$ and $E_k = A_k - [a_{i,j+1},b_{i,j-1})$.
		For $0 \leq t \leq \tau$,
		let $x(t) = a_{i,j}(1 - \frac{t}{\tau}) + \frac{t}{\tau} a_{i,j+1}$.
		Then for $0 \leq t \leq \tau$,
		\[\gamma_k(t) = A_k - [x(t),b_{i,j-1}) + [x(t),b_{i,j}) - [a_{i,j+1},b_{i,j}),\]
		and in particular $\gamma_k(0) = D_k$ and $\gamma_k(\tau) = E_k$.
		See \cref{fig:stability}, {upper right} and lower right.
		For $0 \leq s \leq t < \tau$, 
		\begin{align*}W_{1,1}&(\gamma_k(s),\gamma_k(t))  = \\
			W_{1,1}&(A_k + [x(t),b_{i,j-1}) + [x(s),b_{i,j}) + [a_{i,j+1},b_{i,j}),
			A_k + [x(s),b_{i,j-1}) + [x(t),b_{i,j}) + [a_{i,j+1},b_{i,j}))\text{.}
		\end{align*}
		For visualizations of the couplings, see \cref{fig:case1}. This distance is realized by a coupling that matches identical points and matches  $[x(t),b_{i,j-1})$ with  $[x(s),b_{i,j-1})$ and
		$[x(s),b_{i,j})$ with $[x(t),b_{i,j})$.
		We obtain a distance of $2 (x(t)-x(s)) = 2(t-s)$.
		For $0 \leq s < \tau$, 
		\begin{equation*}W_{1,1}(\gamma_k(s),\gamma_k(\tau)) = 
			W_{1,1}(A_k + [x(s),b_{i,j}) + [a_{i,j+1},b_{i,j-1}),
			A_k  + [x(s),b_{i,j-1}) + [a_{i,j+1},b_{i,j}))\text{.}\end{equation*} 
		This distance is realized by a coupling that matches identical points and matches  $[a_{i,j+1},b_{i,j-1})$ with $[x(s),b_{i,j-1})$ and
		$[x(s),b_{i,j})$ with $[a_{i,j+1},b_{i,j})$.
		We obtain a distance of $2(a_{i,j+1}-x(s)) = 2 (\tau-s)$.                       
		
		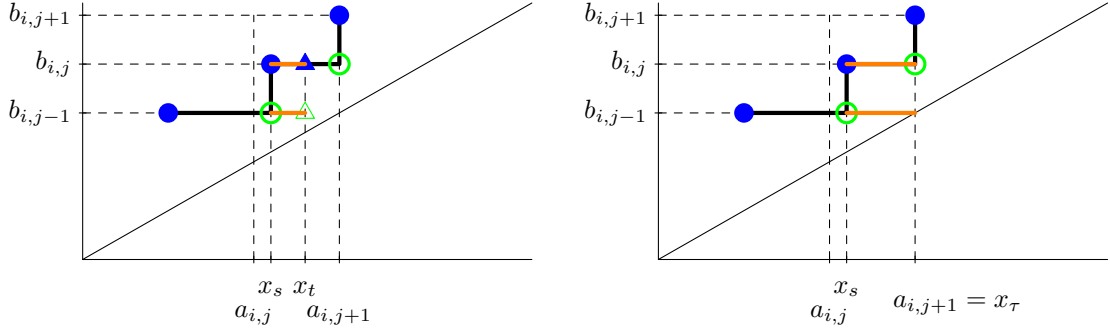
\begin{figure}
			\centering
			\begin{tikzpicture}[line cap=round,line join=round,>=triangle 45,x=1.75cm,y=1.0cm,scale=0.65]
			\draw[color=black] (0,0) -- (5.25,0);
			\draw[color=black] (0,0) -- (0,5.25);
			\draw[color=black] (0,0) -- (5.25,5.25);
			
			\foreach \x in {2,2.2,2.6,3}
			\draw[shift={(\x,0)}] (0pt,2pt) -- (0pt,-2pt);
			\draw[shift={(2,-.3)}] node[below,inner sep=1em] {\footnotesize $a_{i,j}$};
			\draw[shift={(2.2,.2)}] node[below,inner sep=1em] {\footnotesize $x_s$};
			\draw[color=black,dashed] (2.2,0) -- (2.2,3);
			\draw[shift={(2.6,.2)}] node[below,inner sep=1em] {\footnotesize $x_t$};
			\draw[color=black,dashed] (2.6,0) -- (2.6,4);
			\draw[color=black,dashed] (2,0) -- (2,5);				
			\draw[shift={(3,-.3)}] node[below,inner sep=1em] {\footnotesize $a_{i,j+1}$};
			\draw[color=black,dashed] (3,0) -- (3,5);

			\foreach \y in {3,4,5}
			\draw[shift={(0,\y)},color=black] (2pt,0pt) -- (-2pt,0pt);
			\draw[shift={(0,3)}] node[left] {\footnotesize $b_{i,j-1}$};
			\draw[color=black,dashed] (0,3) -- (2,3);				
			\draw[shift={(0,4)}] node[left] {\footnotesize $b_{i,j}$};
			\draw[color=black,dashed] (0,4) -- (2.2,4);	
			\draw[shift={(0,5)}] node[left] {\footnotesize $b_{i,j+1}$};
			\draw[color=black,dashed] (0,5) -- (3,5);				
			
			\draw[color=black,ultra thick] (1,3) -- (2.2,3) -- (2.2,4) -- (3,4) -- (3,5);

			\fill [color=blue] (1,3) circle (2mm);
			\draw [very thick,green=blue] (2.2,3) circle (2mm);
			\fill [color=blue] (2.2,4) circle (2mm);
			\draw [mark=triangle,mark options={color=green},mark size=2.2mm] plot coordinates {(2.6,3)};
			\draw [mark=triangle*,mark options={color=blue},mark size=2.2mm] plot coordinates {(2.6,4)};
			\draw [very thick,green=blue] (3,4) circle (2mm);
			\fill [color=blue] (3,5) circle (2mm);
			
			\draw[color=orange,ultra thick] (2.2,4) -- (2.6,4);
			\draw[color=orange,ultra thick] (2.2,3) -- (2.6,3);
			\end{tikzpicture}
			\quad
			\begin{tikzpicture}[line cap=round,line join=round,>=triangle 45,x=1.75cm,y=1.0cm,scale=0.65]
			\draw[color=black] (0,0) -- (5.25,0);
			\draw[color=black] (0,0) -- (0,5.25);
			\draw[color=black] (0,0) -- (5.25,5.25);
			
			\foreach \x in {2,2.2,3}
			\draw[shift={(\x,0)}] (0pt,2pt) -- (0pt,-2pt);
			\draw[shift={(2,-.3)}] node[below,inner sep=1em] {\footnotesize $a_{i,j}$};
			\draw[shift={(2.2,.2)}] node[below,inner sep=1em] {\footnotesize $x_s$};
			\draw[color=black,dashed] (2.2,0) -- (2.2,3);
			\draw[color=black,dashed] (2,0) -- (2,5);				
			\draw[shift={(3.5,-.1)}] node[below,inner sep=1em] {\footnotesize $a_{i,j+1} = x_\tau$};
			\draw[color=black,dashed] (3,0) -- (3,5);				    									
			
			\foreach \y in {3,4,5}
			\draw[shift={(0,\y)},color=black] (2pt,0pt) -- (-2pt,0pt);
			\draw[shift={(0,3)}] node[left] {\footnotesize $b_{i,j-1}$};
			\draw[color=black,dashed] (0,3) -- (2,3);				
			\draw[shift={(0,4)}] node[left] {\footnotesize $b_{i,j}$};
			\draw[color=black,dashed] (0,4) -- (2.2,4);	
			\draw[shift={(0,5)}] node[left] {\footnotesize $b_{i,j+1}$};
			\draw[color=black,dashed] (0,5) -- (3,5);				
			
			\draw[color=black,ultra thick] (1,3) -- (2.2,3) -- (2.2,4) -- (3,4) -- (3,5);

			\fill [color=blue] (1,3) circle (2mm);
			\draw [very thick,green=blue] (2.2,3) circle (2mm);
			\fill [color=blue] (2.2,4) circle (2mm);
			\draw [very thick,green=blue] (3,4) circle (2mm);
			\fill [color=blue] (3,5) circle (2mm);
			
			\draw[color=orange,ultra thick] (2.2,4) -- (3,4);
			\draw[color=orange,ultra thick] (2.2,3) -- (3,3);
			\end{tikzpicture}
			\caption{Couplings for case (1). Circles indicate $\gamma_k(s)$ and triangles $\gamma_k(t)$. Solid blue points evaluate to +1 and hollow green points evaluate to -1. Orange lines indicate coupled points. All points without orange lines are coupled to themselves.\label{fig:case1}}
		\end{figure}
		
		In case (4),
		let $A_k$ be the points in common in $D_k$ and $E_k$.
		Then $D_k = A_k - [a_{i,j},b_{i,j-1}) + [a_{i,j},b_{i,j})$ and
		$E_k = A_k + [b_{i,j-1},b_{i,j})$.
		For $0 \leq t \leq \tau$,
		let $x(t) = a_{i,j}(1 - \frac{t}{\tau}) + \frac{t}{\tau} b_{i,j-1}$.
		Then for $0 \leq t < \tau$,
		\[\gamma_k(t) = A_k - [x(t),b_{i,j-1}) + [x(t),b_{i,j})\]
		(with $\gamma_k(0) = D_k$) and $\gamma_k(\tau) = E_k$.
		For $0 \leq s \leq t < \tau$, 
		\begin{equation*}W_{1,1}(\gamma_k(s),\gamma_k(t)) = 
			W_{1,1}(A_k + [x(t),b_{i,j-1}) + [x(s),b_{i,j}), A_k + [x(s),b_{i,j-1}) + [x(t),b_{i,j}))\text{.}\end{equation*}
		This distance is realized by a coupling that matches identical points and matches  $[x(t),b_{i,j-1})$ with  $[x(s),b_{i,j-1})$ and
		$[x(s),b_{i,j})$ with $[x(t),b_{i,j})$.
		For visualizations of the couplings, see \cref{fig:case2}.
		We obtain a distance of $2 (x(t)-x(s)) = 2 (t-s)$.
		For $0 \leq s < \tau$, 
		\begin{equation*}W_{1,1}(\gamma_k(s),\gamma_k(\tau)) = 
			W_{1,1}(A_k + [x(s),b_{i,j}), A_k + [b_{i,j-1},b_{i,j}) + [x(t),b_{i,j-1}) )\text{.}\end{equation*}
		This distance is realized by a coupling that matches identical points and matches  $[x(s),b_{i,j})$ with  $[b_{i,j-1},b_{i,j})$ and
		$[x(s),b_{i,j-1})$ with $[b_{i,j-1},b_{i,j-1})$.
		We obtain a distance of $2(b_{i,j-1}-x(s)) = 2 (\tau-s)$.
		
		Case (2) is similar to case (1). Case (3) is similar to case (4) but easier. Case (5) is similar to case (3) but easier still. Therefore $W_{1,1}(D_k,E_k) \leq 2 \tau$.

		\begin{figure}
			\centering
			\begin{tikzpicture}[line cap=round,line join=round,>=triangle 45,x=1.75cm,y=1.0cm,scale=0.65]
			\draw[color=black] (0,0) -- (5.25,0);
			\draw[color=black] (0,0) -- (0,5.25);
			\draw[color=black] (0,0) -- (5.25,5.25);
			
			\foreach \x in {2,2.2,2.6,3}
			\draw[shift={(\x,0)}] (0pt,2pt) -- (0pt,-2pt);
			\draw[shift={(2,-.3)}] node[below,inner sep=1em] {\footnotesize $a_{i,j}$};
			\draw[shift={(2.2,.2)}] node[below,inner sep=1em] {\footnotesize $x_s$};
			\draw[color=black,dashed] (2.2,0) -- (2.2,3);
			\draw[shift={(2.6,.2)}] node[below,inner sep=1em] {\footnotesize $x_t$};
			\draw[shift={(3,-.3)}] node[below,inner sep=1em] {\footnotesize $b_{i,j-1}$};			
			\draw[color=black,dashed] (2.6,0) -- (2.6,4);
			\draw[color=black,dashed] (2,0) -- (2,4);				
			\draw[color=black,dashed] (3,0) -- (3,4);

			\foreach \y in {3,4}
			\draw[shift={(0,\y)},color=black] (2pt,0pt) -- (-2pt,0pt);
			\draw[shift={(0,3)}] node[left] {\footnotesize $b_{i,j-1}$};
			\draw[color=black,dashed] (0,3) -- (3,3);				
			\draw[shift={(0,4)}] node[left] {\footnotesize $b_{i,j}$};
			\draw[color=black,dashed] (0,4) -- (2.2,4);

			\draw[color=black,ultra thick] (1,3) -- (2.2,3) -- (2.2,4) -- (3,4) -- (4,4);

			\fill [color=blue] (1,3) circle (2mm);
			\draw [very thick,green=blue] (2.2,3) circle (2mm);
			\fill [color=blue] (2.2,4) circle (2mm);
			\draw [mark=triangle,mark options={color=green},mark size=2.2mm] plot coordinates {(2.6,3)};
			\draw [mark=triangle*,mark options={color=blue},mark size=2.2mm] plot coordinates {(2.6,4)};

			\draw[color=orange,ultra thick] (2.2,4) -- (2.6,4);
			\draw[color=orange,ultra thick] (2.2,3) -- (2.6,3);
			\end{tikzpicture}
			\quad
			\begin{tikzpicture}[line cap=round,line join=round,>=triangle 45,x=1.75cm,y=1.0cm,scale=0.65]
			\draw[color=black] (0,0) -- (5.25,0);
			\draw[color=black] (0,0) -- (0,5.25);
			\draw[color=black] (0,0) -- (5.25,5.25);
			
			\foreach \x in {2,2.2,3}
			\draw[shift={(\x,0)}] (0pt,2pt) -- (0pt,-2pt);
			\draw[shift={(2,-.3)}] node[below,inner sep=1em] {\footnotesize $a_{i,j}$};
			\draw[shift={(2.2,.2)}] node[below,inner sep=1em] {\footnotesize $x_s$};
			\draw[color=black,dashed] (2.2,0) -- (2.2,3);
			\draw[shift={(3.2,-.3)}] node[below,inner sep=1em] {\footnotesize $b_{i,j-1} = x_\tau$};
			\draw[color=black,dashed] (2,0) -- (2,4);
			\draw[color=black,dashed] (3,0) -- (3,4);

			\foreach \y in {3,4}
			\draw[shift={(0,\y)},color=black] (2pt,0pt) -- (-2pt,0pt);
			\draw[shift={(0,3)}] node[left] {\footnotesize $b_{i,j-1}$};
			\draw[color=black,dashed] (0,3) -- (2,3);				
			\draw[shift={(0,4)}] node[left] {\footnotesize $b_{i,j}$};
			\draw[color=black,dashed] (0,4) -- (2.2,4);

			\draw[color=black,ultra thick] (1,3) -- (2.2,3) -- (2.2,4) -- (3,4) -- (4,4);

			\fill [color=blue] (1,3) circle (2mm);
			\draw [very thick,green=blue] (2.2,3) circle (2mm);
			\fill [color=blue] (2.2,4) circle (2mm);
			\draw [mark=triangle*,mark options={color=blue},mark size=2.2mm] plot coordinates {(3,4)};

			\draw[color=orange,ultra thick] (2.2,4) -- (3,4);
			\draw[color=orange,ultra thick] (2.2,3) -- (3,3);
			\end{tikzpicture}
			\caption{Couplings for case (4). Circles indicate $\gamma_k(s)$ and triangles $\gamma_k(t)$. Orange lines indicate coupled points. All points without orange lines are coupled to themselves.\label{fig:case2}}
		\end{figure}
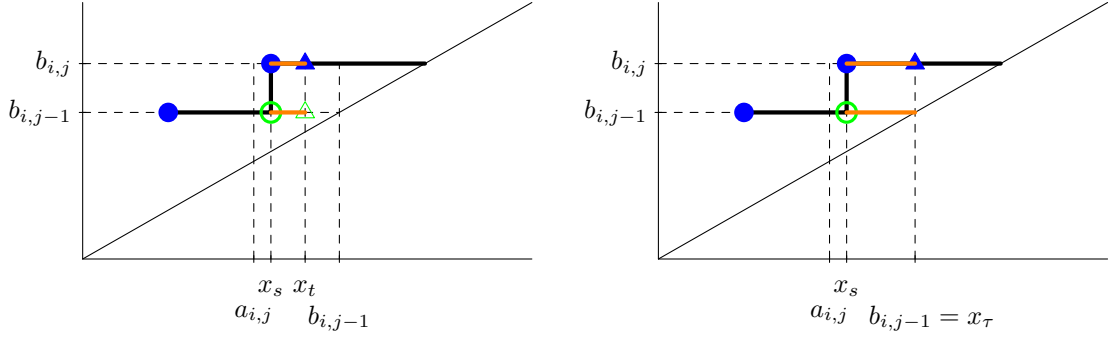
		
		Since $\rank(M),\rank(N) \leq K$, for $k > K$, $\rank_k(M) = \rank_k(N) = 0$ and thus $D_k = E_k = 0$. In addition, since $\rank(M), \rank(N) \leq K$, $D_K$ and $E_K$ have only positive points. Thus the coordinate geodesics only move one point in $D_K$ and $E_K$ at a time. Hence $W_{1,1}(D_K,E_K) \leq W_{1,1}(D,E)$.
		
		We first exhibit that the Theorem's bounds are attained for an example with $K=3$, then move to the general case.
		See \cref{fig:bounds}.
		Let $A = [1,7) + [2,8)$ and consider $D = A + [3,9)$ and $E = A + [4,9)$. Applying \cref{prop:metric} gives $W_{1,1}(D,E) = 1$. Also, $D_1 = A_1 - [3,8) + [3,9)$ and $E_1 = A_1 - [4,8) + [4,9)$, so $W_{1,1}(D_1,E_1) = 2$. The minimum cost coupling pairs $[3,8)$ to $[4,8)$ and $[3,9)$ to $[4,9)$. Similarly, $D_2 = A_2 - [3,7) + [3,8)$, $E_2 = A_2 - [4,7) + [4,8)$, and $W_{1,1}(D_2,E_2) = 2$. Finally, $D_3 = [3,7)$ and $E_3 = [4,7)$ so $W_{1,1}(D_3,E_3) = 1$. 
		
		\begin{figure}
			\centering
			\begin{tikzpicture}[line cap=round,line join=round,>=triangle 45,x=1.0cm,y=1.0cm,scale=0.65]
			\draw[color=black] (0,0) -- (7.25,0);
			\draw[color=black] (0,0) -- (0,7.25);
			\draw[color=black] (0,0) -- (7.25,7.25);
			
			\foreach \x in {1,2,3,4}
			\draw[shift={(\x,0)}] (0pt,2pt) -- (0pt,-2pt);
			\foreach \x in {1,2,3,4}
			\draw[shift={(\x,0)},color=black] node[below] {\footnotesize $\x$};
			
			\foreach \y in {7,8,9}
			\draw[shift={(0,\y-2)},color=black] (2pt,0pt) -- (-2pt,0pt);
			\foreach \y in {7,8,9}
			\draw[shift={(0,\y-2)},color=black] node[left] {\footnotesize $\y$};
			
			\draw [mark=square*,mark options={color=blue},mark size=4pt] plot coordinates {(1,5)};
			\draw [mark=square*,mark options={color=blue},mark size=4pt] plot coordinates {(2,6)};
			\fill [color=blue] (3,7) circle (2mm);
			\draw [mark=triangle*,mark options={color=blue},mark size=2.2mm] plot coordinates {(4,7)};
			
			\end{tikzpicture}	
			
			\begin{tikzpicture}[line cap=round,line join=round,>=triangle 45,x=1.0cm,y=1.0cm,scale=0.5]
			\draw[color=black] (0,0) -- (7.25,0);
			\draw[color=black] (0,0) -- (0,7.25);
			\draw[color=black] (0,0) -- (7.25,7.25);
			
			\foreach \x in {1,2,3,4}
			\draw[shift={(\x,0)}] (0pt,2pt) -- (0pt,-2pt);
			\foreach \x in {1,2,3,4}
			\draw[shift={(\x,0)},color=black] node[below] {\footnotesize $\x$};
			
			\foreach \y in {7,8,9}
			\draw[shift={(0,\y-2)},color=black] (2pt,0pt) -- (-2pt,0pt);
			\foreach \y in {7,8,9}
			\draw[shift={(0,\y-2)},color=black] node[left] {\footnotesize $\y$};
			
			\draw [mark=square*,mark options={color=blue},mark size=4pt] plot coordinates {(1,5)};
			\draw [mark=square,mark options={color=green},mark size=4pt] plot coordinates {(2,5)};
			\draw [mark=square*,mark options={color=blue},mark size=4pt] plot coordinates {(2,6)};	
			\draw [very thick,green=blue] (3,6) circle (2mm);			
			\fill [color=blue] (3,7) circle (2mm);
			\draw [mark=triangle*,mark options={color=blue},mark size=2.2mm] plot coordinates {(4,7)};
			\draw [mark=triangle,mark options={color=green},mark size=2.2mm] plot coordinates {(4,6)};
			
			\draw[color=orange,ultra thick] (3,6) -- (4,6);
			\draw[color=orange,ultra thick] (3,7) -- (4,7);
			
			\end{tikzpicture}
			\begin{tikzpicture}[line cap=round,line join=round,>=triangle 45,x=1.0cm,y=1.0cm,scale=0.5]
			\draw[color=black] (0,0) -- (7.25,0);
			\draw[color=black] (0,0) -- (0,7.25);
			\draw[color=black] (0,0) -- (7.25,7.25);
			
			\foreach \x in {1,2,3,4}
			\draw[shift={(\x,0)}] (0pt,2pt) -- (0pt,-2pt);
			\foreach \x in {1,2,3,4}
			\draw[shift={(\x,0)},color=black] node[below] {\footnotesize $\x$};
			
			\foreach \y in {7,8,9}
			\draw[shift={(0,\y-2)},color=black] (2pt,0pt) -- (-2pt,0pt);
			\foreach \y in {7,8,9}
			\draw[shift={(0,\y-2)},color=black] node[left] {\footnotesize $\y$};
			
			\draw [mark=square*,mark options={color=blue},mark size=4pt] plot coordinates {(2,5)};
			\draw [very thick,green=blue] (3,5) circle (2mm);
			\fill [color=blue] (3,6) circle (2mm);
			\draw [mark=triangle*,mark options={color=blue},mark size=2.2mm] plot coordinates {(4,6)};
			\draw [mark=triangle,mark options={color=green},mark size=2.2mm] plot coordinates {(4,5)};								
			
			\draw[color=orange,ultra thick] (3,5) -- (4,5);
			\draw[color=orange,ultra thick] (3,6) -- (4,6);
			
			\end{tikzpicture}		
			\begin{tikzpicture}[line cap=round,line join=round,>=triangle 45,x=1.0cm,y=1.0cm,scale=0.5]
			\draw[color=black] (0,0) -- (7.25,0);
			\draw[color=black] (0,0) -- (0,7.25);
			\draw[color=black] (0,0) -- (7.25,7.25);
			
			\foreach \x in {1,2,3,4}
			\draw[shift={(\x,0)}] (0pt,2pt) -- (0pt,-2pt);
			\foreach \x in {1,2,3,4}
			\draw[shift={(\x,0)},color=black] node[below] {\footnotesize $\x$};
			
			\foreach \y in {7,8,9}
			\draw[shift={(0,\y-2)},color=black] (2pt,0pt) -- (-2pt,0pt);
			\foreach \y in {7,8,9}
			\draw[shift={(0,\y-2)},color=black] node[left] {\footnotesize $\y$};
			
			\fill [color=blue] (3,5) circle (2mm);
			\draw [mark=triangle*,mark options={color=blue},mark size=2.2mm] plot coordinates {(4,5)};			
			
			\draw[color=orange,ultra thick] (3,5) -- (4,5);
			
			\end{tikzpicture}
			\caption{{\bf Top:} Example with $K=3$ that attains the bounds in \cref{thm:stability}. The persistence diagram $A$ is given by the squares, persistence diagram $D$ is given by the squares and the circle, and persistence diagram $E$ is given by the squares and the triangles. {\bf Bottom:} The graded persistence diagrams $A_k,D_k$, and $E_k$ for $k=1,2,3$ (left-to-right). Orange lines indicate coupled points. All points without orange lines are coupled to themselves.}\label{fig:bounds} 
		\end{figure}
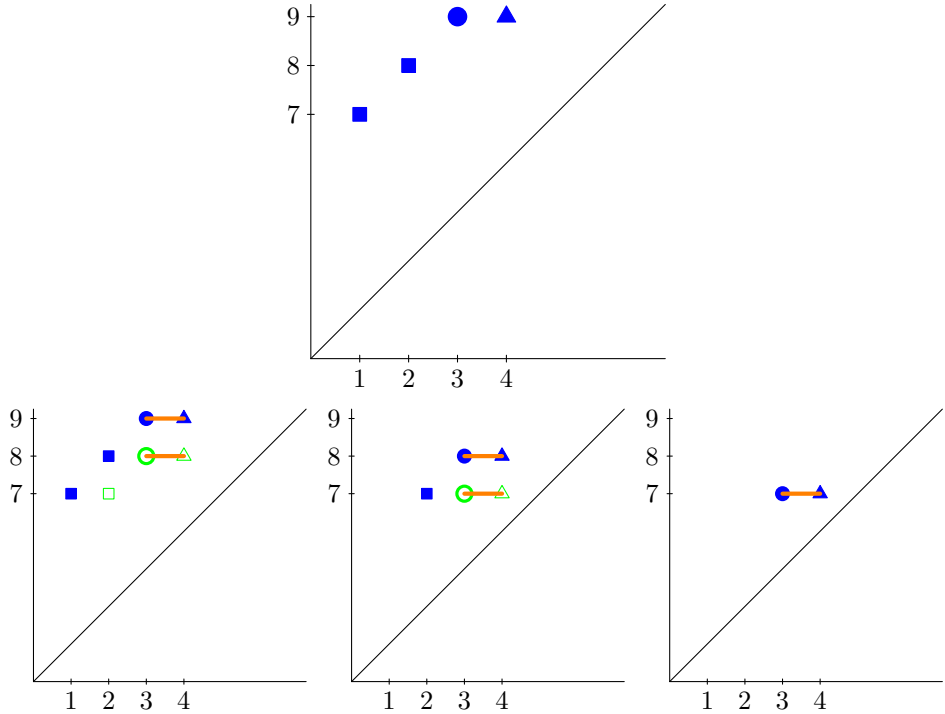
		
		In the general case, let $A =[1,2K+1)+[2,2K+2)+\dots+[K-1,3K-1)$,
		$D = A + [K,3K)$, and $E = A + [K+1,3K)$. Then $D$ and $E$ only differ by $[K,3K)$ and $[K+1,3K)$, so $W_{1,1}(D,E)=1$. For $1 \leq k < K$, $D_k = A_k - [K,3K-k) + [K,3K+1-k)$ and $E_k = A_k - [K+1,3K-k) + [K+1,3K+1-k)$, so $W_{1,1}(D_k,E_k) = 2$. Finally $D_K = [K,2K+1)$ and $E_K = [K+1,2K+1)$. So $W_{1,1}(D_K,E_K) = 1$.
	\end{proof}
	
	\begin{remark}
		\cref{thm:stability} may be combined with Skraba and Turner's recent Wasserstein stability theorems~\cite{Skraba:2020} to obtain Wasserstein stability of the graded persistence diagrams in various settings.
	\end{remark}
	
	Let $M$ and $N$ be persistence modules with
	persistence diagrams $D,E: \Dgm \to \mz_{\geq 0}$ and 
	$k$-th graded persistence diagrams $D_k,E_k: \Dgm \to \mz$ for $k \geq 1$.
	Let $K$ be the maximum of $\rank(M)$ and $\rank(N)$.

	\begin{theorem} \label{thm:bounds}
		We have
		\[ W_{1,1}(D,E) \leq \sum_{k=1}^K W_{1,1}(D_k,E_k) \leq (2K-1) W_{1,1}(D,E)
		\]
		and these bounds are sharp.
	\end{theorem}

	\begin{proof}
		The right hand inequality and the fact that it is sharp are an immediate consequence of \cref{thm:stability}.
		Next, we prove the left hand inequality.
		By \cref{thm:consistency}, $D = \sum_{k=1}^K D_k$ and $E = \sum_{k=1}^K E_k$.
		Let $A,A',B,B':\Dgm\to\mz$ be finitely supported functions.
		Then by the triangle inequality and \cref{prop:metric},
		$W_{1,1}(A+A',B+B') \leq W_{1,1}(A+A',B+A') + W_{1,1}(B+A',B+B')
		= W_{1,1}(A,B) + W_{1,1}(A',B')$.
		By induction, we obtain the left inequality.
		To see that the left inequality is sharp, take $M$ and $N$ to be interval modules.
	\end{proof}

	\subsection*{Acknowledgments}
	
	The authors would like to thank the anonymous referees whose many comments considerably improved our manuscript. This research was partially supported by the Southeast Center for Mathematics and Biology, an NSF-Simons Research Center for Mathematics of Complex Biological Systems, under National Science Foundation Grant No. DMS-1764406 and Simons Foundation Grant No. 594594.
	This material is based upon work supported by, or in part by, the Army Research Laboratory and the Army Research Office under contract/grant number
	W911NF-18-1-0307.
	

\end{document}